\documentclass[11pt]{amsart}
\usepackage[all, import, 2cell]{xy} \UseAllTwocells 
\usepackage{graphicx} 
\usepackage{amsmath}
\usepackage{mathrsfs}
\usepackage{amsfonts}
\usepackage{mathabx}
\usepackage{amsthm} 
\usepackage[usenames]{color} 
\usepackage{hyperref}
\usepackage{bm}
\numberwithin{equation}{section}
\newtheorem{theorem}{Theorem}[section]

\newtheorem{lemma}[theorem]{Lemma}

\newtheorem{prop}[theorem]{Proposition}
\newtheorem{corollary}[theorem]{Corollary}

\theoremstyle{definition}
\newtheorem{defn}[theorem]{Definition}

\newtheorem{example}[theorem]{Example}

\newtheorem{remark}[theorem]{Remark}

\newcommand{\nc}{\newcommand}
\nc{\DMO}{\DeclareMathOperator}	

\nc{\newnotation}{\nomenclature}
\nc{\wrap}{\cW}
\nc{\wf}{\mathsf{WF}}
\nc{\inff}{\mathsf{InfF}}
\nc{\shv}{\mathsf{Shv}}
\nc{\pwr}{\mathsf{PWr}}
\nc{\Cob}{\mathsf{Cob}}
\nc{\mul}{\mathsf{Mul}}
\nc{\fat}{\mathsf{fat}}
\nc{\cob}{\mathsf{Cob}}
\nc{\coh}{\mathsf{Coh}}
\nc{\sets}{\mathsf{Sets}}
\nc{\Exit}{\mathsf{Exit}}
\nc{\near}{\mathsf{near}}
\nc{\sing}{\mathsf{Sing}}
\nc{\symp}{\mathsf{Symp}}
\nc{\perf}{\mathsf{Perf}}
\nc{\calg}{\mathsf{CAlg}}
\nc{\ssets}{\mathsf{sSets}}
\nc{\cmpct}{\mathsf{cmpct}}
\nc{\pwrap}{\mathsf{PWrap}}
\nc{\coder}{\mathsf{Coder}}
\nc{\bimod}{\mathsf{Bimod}}
\nc{\grmod}{\mathsf{GrMod}}
\nc{\spaces}{\mathsf{Spaces}}
\nc{\pwrms}{\mathsf{PWrFuk}_{M,S}}
\nc{\pwrmf}{\mathsf{PWrFuk}_{M,F}}
\nc{\pwrapmf}{\mathsf{PWrFuk}_{M,F}}
\nc{\fuk}{\mathsf{Fukaya}}
\nc{\infwr}{\mathsf{InfWr}}
\nc{\fukaya}{\mathsf{Fukaya}}
\nc{\autml}{\mathsf{Aut}_{M,\Lambda}}
\nc{\fukml}{\mathsf{Fukaya}_{M,\Lambda}}
\nc{\fukmle}{\mathsf{Fukaya}_{M,\Lambda,\epsilon}}
\nc{\fukmod}{\fun(\mathsf{Fukaya}_{M,\Lambda}^{\op},\chain_\ZZ)}
\nc{\lag}{\mathsf{Lag}}
\nc{\lagm}{\lag_M}
\nc{\lagml}{\lag_{M,\Lambda}} 
\nc{\lagmle}{\lag_{M,\Lambda,\epsilon}}
\nc{\fun}{\mathsf{Fun}}
\nc{\vect}{\mathsf{Vect}}
\nc{\chain}{\mathsf{Chain}}
\nc{\morita}{\mathsf{Morita}}
\nc{\wrfuk}{\mathsf{WrFukaya}}
\nc{\pwrfuk}{\mathsf{PWrFukaya}}
\nc{\inffuk}{\mathsf{InfFuk}}
\nc{\pwrfukml}{\mathsf{PWrFukaya}_{M,\Lambda}}
\nc{\inffukml}{\mathsf{InfFuk}_{M,\Lambda}}
\nc{\corres}{\mathsf{Corres}}
\nc{\cat}{\mathsf{Cat}}
\nc{\fukep}{\fukaya_\Lambda(M,\epsilon)}
\nc{\fukepop}{\fukaya_\Lambda(M,\epsilon)^{\op}}
\nc{\lagep}{\lag_\Lambda(M,\epsilon)}
\DMO{\cyl}{cyl} 
\nc{\dbcoh}{D^b\mathsf{Coh}}
\nc{\corr}{\mathsf{Corr}}
\nc{\sympdiscrete}{\underline{\mathsf{Symp}}}
\nc{\cdgadiscrete}{\underline{\mathsf{cdga}}}

\nc{\compact}{\mathsf{Compact}}
\nc{\categories}{\mathsf{Cats}}

\nc{\ainfty}{\mathsf{A}_\infty}
\nc{\inftycat}{\mathcal{C}\!\operatorname{at}_\infty}
\nc{\Ainftycat}{\mathcal{C}\!\operatorname{at}_{A_\infty}}
\nc{\ainftycat}{\mathcal{C}\!\operatorname{at}_{A_\infty}}
\nc{\stablecat}{\mathcal{C}\!\operatorname{at}_\infty^{\Ex}}

\DMO{\fib}{fib}
\DMO{\conf}{Conf}
\DMO{\chains}{Chains}
\DMO{\cochains}{Cochains}
\DMO{\cone}{Cone}
\DMO{\ran}{Ran}
\DMO{\leg}{Leg}
\DMO{\imm}{imm}
\DMO{\adj}{adj}
\DMO{\cube}{Cube}
\DMO{\flow}{Flow}
\DMO{\floer}{Floer}
\DMO{\maps}{Maps}
\DMO{\grph}{graph}
\DMO{\exact}{exact}
\DMO{\Decomp}{Decomp}
\DMO{\decomp}{Decomp}
\DMO{\yoneda}{Yoneda}
\DMO{\hamspace}{Ham}
\DMO{\sympspace}{Symp}
\DMO{\holomaps}{Holomaps}
\DMO{\comp}{Comp}
\DMO{\crit}{Crit}
\DMO{\test}{{test}}
\DMO{\sign}{sign}
\DMO{\topp}{top}
\DMO{\indx}{Index}
\DMO{\Break}{Break} 
\DMO{\zero}{zero} 
\DMO{\ob}{Ob}
\DMO{\gr}{Gr} 
\DMO{\Gr}{Gr} 
\DMO{\cl}{Cl} 
\DMO{\grlag}{GrLag}
\DMO{\Pin}{Pin}
\DMO{\Graph}{Graph}
\DMO{\pin}{Pin}
\DMO{\gap}{Gap}
\DMO{\Ex}{Ex}
\DMO{\id}{id}
\DMO{\End}{End}
\DMO{\sym}{Sym} 
\DMO{\aut}{Aut}
\DMO{\DK}{DK} 
\DMO{\poly}{poly} 
\DMO{\diff}{Diff}
\DMO{\coll}{coll}
\DMO{\dist}{dist} 
\DMO{\coker}{coker} 
\nc{\kernel}{\ker} 
\DMO{\sspan}{span}
\DMO{\hocolim}{hocolim}	
\DMO{\holim}{holim}
\DMO{\sk}{sk}

\DMO{\ev}{ev}
\DMO{\im}{im}
\DMO{\ho}{ho}
\DMO{\fin}{fin}
\DMO{\ret}{Ret}
\DMO{\ham}{Ham}
\DMO{\con}{con}
\DMO{\leaf}{leaf}
\DMO{\supp}{supp}
\DMO{\edge}{edge}
\DMO{\colim}{colim}
\DMO{\edges}{edges}
\DMO{\Image}{image}
\DMO{\roots}{roots}
\DMO{\height}{height}
\DMO{\finmod}{FinMod}
\DMO{\leaves}{leaves}
\DMO{\planar}{planar}
\DMO{\vertices}{vertices}

\nc{\lagg}{\lag^{\cG}}
\nc{\iso}{\mathsf{Iso}}
\nc{\bsc}{\mathcal{B}\mathsf{sc}}
\nc{\Set}{\mathsf{Set}}
\nc{\ass}{\mathsf{ \bf Ass}}
\nc{\Mod}{\mathsf{Mod}}
\nc{\exit}{\mathsf{Exit}}
\nc{\sset}{\mathsf{sSet}}
\nc{\liou}{\mathsf{Liou}}
\nc{\poset}{\mathsf{Poset}}
\nc{\trno}{T^*\RR^n_{\geq 0}}
\nc{\spectra}{\mathsf{Spectra}}
\nc{\tensorfin}{\tensor^{\fin}}
\nc{\lagptg}{\lag_{pt,pt}^{\cG}}
\nc{\Fin}{\mathcal{F}\mathsf{in}}
\nc{\lagnl}{\lag_{N,\Lambda}}
\nc{\lagmlg}{\lag_{M,\Lambda}^{\cG}}
\nc{\lagsplit}{\lag^{\mathsf{split}}}
\nc{\lagktimes}{(\lag^{\dd k})^\times}
\nc{\lagplanar}{\lag^{\times,\planar}}

\nc{\dlambda}{d^\Lambda}
\nc{\smsh}{\wedge}
\nc{\un}{\underline}
\nc{\xto}{\xrightarrow}
\nc{\xra}{\xto}
\nc{\tensor}{\otimes}
\nc{\del}{\partial}
\nc{\dd}{\diamond}
\nc{\tri}{\triangle}
\nc{\bb}{\Box}
\nc{\into}{\hookrightarrow}
\nc{\contains}{\supset}
\nc{\immto}{\looparrowright}
\nc{\transverse}{\pitchfork}
\nc{\uncirc}{\underline{\circ}}
\nc{\Jbar}{\overline{J}}
\nc{\Fbar}{\overline{F}}
\nc{\delbar}{\overline{\del}}
\nc{\thetabar}{\overline{\theta}}
\nc{\omegabar}{\overline{\omega}}

\nc{\colldiff}{\diff^{\del}} 
\nc{\trbar}{\overline{T^*\RR}}
\nc{\tr}{T^*\RR}
\nc{\tsa}{Ts\cA}
\nc{\tsb}{Ts\cB}
\nc{\cmbar}{\overline{\cM}}
\nc{\crbar}{\overline{\cR}}

\nc{\vece}{ {\vec \epsilon}}	
\nc{\vecd}{ {\vec \delta}}
\nc{\ov}{\overline}
\DMO{\op}{op}
\nc{\opp}{ ^{\op}}

\nc{\eqn}{\begin{equation}}
\nc{\eqnn}{\begin{equation*}}
\nc{\eqnnd}{\end{equation*}}
\nc{\eqnd}{\end{equation}}
\nc{\enum}{\begin{enumerate}}
\nc{\enumd}{\end{enumerate}}

\def\cA{\mathcal A}\def\cB{\mathcal B}\def\cC{\mathcal C}\def\cD{\mathcal D}
\def\cF{\mathcal F}\def\cG{\mathcal G}\def\cH{\mathcal H}
\def\cJ{\mathcal J}\def\cL{\mathcal L}
\def\cM{\mathcal M}\def\cO{\mathcal O}
\def\cR{\mathcal R}
\def\cV{\mathcal V}\def\cW{\mathcal W}

\def\RR{\mathbb R}

\def\ZZ{\mathbb Z}

\def\sA{\mathsf A}

\DMO{\prim}{Prim}

\nc{\LS}{\textcolor{magenta}}
\nc{\w}{\wedge}
\nc{\dpp}{\partial_+}
\nc{\dpm}{\partial_-}
\nc{\dl}{d^{\Lambda}}
\nc{\om}{\omega}
\nc{\Om}{\Omega}
\nc{\Alg}{\mathsf{Alg}}
\nc{\cdga}{\mathsf{cdga}}
\nc{\open}{\mathsf{Open}}
\nc{\Span}{\mathsf{Span}}
\nc{\dmfld}{\mathsf{dMfld}}
\nc{\Cospan}{\mathsf{Cospan}}
\DMO{\spec}{Spec}
\nc{\pa}{\partial}
\nc{\paa}{\partial'}
\nc{\pab}{\partial''}

\nc{\thp}{{\theta_{2p+1}}}
\nc{\cpF}{{\cF_p}}
\nc\ddp{{d_+}}
\nc\ddm{{d_-}}
\nc\rstar{{\ast_r}}
\nc\gstar{{\ast_g}}
\def\pip{\Pi^p}
\def\pips{{{\Pi^{p}}^*}}
\nc\FO[1]{F^{p}\Omega^{#1}}
\nc\BFO[1]{\overline{{F^{p}\Omega^{#1}}}}
\def\xa{x_1}
\def\xb{x_2}
\def\xc{x_3}
\def\xd{x_4}
\def\ea{e_1}
\def\eb{e_2}
\def\ec{e_3}
\def\ed{e_4}

\begin{document}

\title{Odd Sphere Bundles, Symplectic Manifolds, and their Intersection Theory}
\author{Hiro Lee Tanaka}
\address{Department of Mathematics, Harvard University, Cambridge, MA 02138}
\email{hirolee@math.harvard.edu}
\author{Li-Sheng Tseng}
\address{Department of Mathematics, University of California, Irvine, CA 92697}
\email{lstseng@math.uci.edu}

\begin{abstract}
Recently, Tsai-Tseng-Yau constructed new invariants of symplectic manifolds: a sequence of $A_\infty$-algebras built of differential forms on the symplectic manifold. We show that these symplectic $A_\infty$-algebras have a simple topological interpretation. Namely, when the cohomology class of the symplectic form is integral, these $A_\infty$-algebras are equivalent to the standard de Rham differential graded algebra on certain odd-dimensional sphere bundles over the symplectic manifold.  From this equivalence, we deduce for a closed symplectic manifold that Tsai-Tseng-Yau's symplectic $A_\infty$-algebras satisfy the Calabi-Yau property, and importantly, that they can be used to define an intersection theory for coisotropic/isotropic chains.  We further demonstrate that these symplectic $A_\infty$-algebras satisfy several functorial properties and lay the groundwork for addressing Weinstein functoriality and invariance in the smooth category.
\end{abstract}



\maketitle

\tableofcontents

\section{Introduction}
Recently, Tsai-Tseng-Yau~\cite{tsai-tseng-yau} introduced
a family of algebras of differential forms for any symplectic manifold $M$.   These algebras---denoted $\cpF(M)$, with parameter $p$ taking integral values from zero to $\dim M /2$---have several notable properties.  First, each $\cpF(M)$ is an $A_\infty$-algebra. Second, the differential of the algebra consists of both first- and second-order differential operators, while the product interestingly involves derivatives.  Finally, their cohomologies, $F^pH^*(M)$, are all finite-dimensional---they were named the {\em $p$-filtered cohomologies} of $M$ for short~\cite{tsai-tseng-yau} .   These filtered cohomology rings intrinsically depend on, and can vary with, the symplectic structure $\om$ \cite{tseng-yau-2, tsai-tseng-yau}.

In this paper, we equate the algebras $\cpF$ (in the $A_\infty$-algebra sense) to a standard de Rham differential graded algebra---not on the original symplectic manifold $M$, but on certain odd-dimensional sphere bundles over $M$. This perspective turns out to give a simple, topological description of $\cpF$, and makes manifest several useful properties of $\cpF$ (including some functorial properties).  

We can motivate the existence of such an alternative description for $\cpF$ on $(M^{2n}, \om)$ from a broader point of view. For a symplectic manifold $M$, the symplectic structure and its powers constitute a distinguished set of forms:
\begin{align*}
\left\{\om, \om^2, \ldots, \om^n\right\} \in \Omega^*(M)\,.
\end{align*}
Special forms also arise in geometry as representatives of characteristics classes of bundles over $M$.  So we consider bundles over $M$ with characteristics classes given by powers of $\om$.

Explicitly, for the distinguished two-form $\om$, consider the line bundle $L$ over $M$ with $c_1 = (1/2\pi)\, \om$.  This $L$ is the well-known prequantum line bundle in geometric quantization, and has an associated circle bundle with Euler class $e=\om$.  More generally, for the $2(p+1)$-form $\om^{p+1}$, we have the associated direct sum vector bundle $L^{\oplus\, p+1}$, and we will focus in particular on the associated closed $S^{2p+1}$ bundle over $M$:
\begin{align}
	\xymatrix{
	S^{2p+1} \ar[r] & E_{p} \ar[d]\\
	& (M^{2n}, \om)
	}
\end{align}
with its Euler class given by $\om^{p+1}$.   
Given that $E_p$ is derived from the symplectic form $\omega$, one can ask how the topology of $E_p$ relates to the symplectic geometry of $M^{2n}$:
\begin{align*}
\xymatrix{
&\fbox{Topological data of $E_p$}\qquad \ar@{<-->}[rr]^-*+\txt{?} & & \qquad \fbox{Symplectic data of $(M^{2n}, \om)$}
}
\end{align*}
\noindent Specifically,
\begin{itemize}
\item[(1)] What symplectic data are encoded in the topology of the odd sphere bundle $E_p$?  
\item[(2)] Can we use the topology of these sphere bundles to gain new insights  into the known symplectic invariants on $(M^{2n}, \om)$?
\end{itemize}

In this paper, we begin to address such questions by analyzing the cohomology and the differential topology of $E_p$.  First, the real cohomology of the sphere bundle---equivalently, its de Rham cohomology---is well-known to fit into the Gysin sequence: 
\begin{align}\label{triangled}
	\xymatrix{
	& H^*(E_p)\ar[dl] &\\
	H^*(M) \ar[rr]^{\w \om^{p+1}}& &H^*(M)  \ar[ul]
	}
\end{align}
where again $\om^{p+1}$ is the Euler class of $E_p$.  This exact triangle identifies $H^*(E_p)$ non-canonically as the direct sum of kernels and cokernels of the $\om^{p+1}$ map between the de Rham cohomologies of $M$.  Coincidentally, exactly the same triangle diagram also appeared in \cite{tsai-tseng-yau} for the symplectic $p$-filtered cohomologies of differential forms, $F^pH^*(M)$; so we can quickly conclude from \eqref{triangled} that 
\begin{align}\label{HEF}
H^*(E_p) &\cong \coker[\om^{p+1}: H^*(M)\to H^*(M)] \oplus \ker[\om^{p+1}: H^*(M) \to H^*(M)] \\
&\cong F^pH^*(M)\,. \nonumber
\end{align}    

Of course, to construct an honest sphere bundle $E_p$, the Euler class $e=\om^{p+1}$ must be an integral class, i.e. an element of $H^{2p+2}(M, \ZZ)$. In the case the cohomology class is not integral, we should instead consider $E_p$ to be a fibration of rational spheres, in the sense of rational homotopy theory. However, from the perspective of the above two questions, the integral condition will often not be relevant.  For instance, the resulting cohomology on $E_p$ still fits into a triangle as in \eqref{HEF}, just as $F^pH^*(M)$ does regardless of the integrality of $\omega$.

Beyond the cohomology of $E_p$, we can also consider the de Rham algebra of $E_p$.  When $[\omega]$ is integral, a standard application of the \v{C}ech-de Rham double complex shows that
	\eqnn
		\Omega(E_p) \simeq \cone(\omega^{p+1})
	\eqnnd
where $\cone(\omega^{p+1})$ is the mapping cone of the map between de Rham chain complexes:
	\eqnn		
	\wedge \omega^{p+1}:	\Omega^\bullet(M)[-2p-2] \to \Omega^\bullet(M)\,.
	\eqnnd
We refer the interested reader to the proof in the Appendix. 

The cone algebra $\cone(\omega^{p+1})$ is a commutative differential graded algebra (cdga).   It is also well-defined for any $\om$ whether integral or not.  At first glance, this cone cdga is very different from the $A_\infty$-algebra $\cpF$, but both fit into the same long exact sequence of cohomology groups (hence have isomorphic cohomology).  This raises the question whether $\cpF$ and $\cone(\omega^{p+1})$ represent two distinct symplectic algebras on $(M, \om)$.   We prove that they are in fact equivalent:

\begin{theorem}\label{thm.cone}	
For each $p$, there exists a natural equivalence of $A_\infty$-algebras
			\eqnn
				\cF_p \simeq \cone(\omega^{p+1})\,.
			\eqnnd
\end{theorem}

Moreover, this equivalence suggests that $\cF_p$ should come with a non-degenerate pairing when $M$ is compact---after all, when $[\omega]$ is integral, $\Omega(E_p)$ has a non-degenerate pairing by Poincar\'e Duality. To that end, we prove:

\begin{theorem}
Each $A_\infty$-algebra $\cF_p$ is Calabi-Yau whenever $M$ is compact.
\end{theorem}

In other words, each $\cF_p$ admits a cyclically invariant pairing, and this theorem holds for all symplectic structures on $M$.

At this point, two natural questions arise from these theorems, whose answers we begin to address in this paper, and which we plan to deepen in future work.

First: Is there an intersection theory associated to each $\cF_p$? After all, common examples of Calabi-Yau algebras are differential forms on compact, oriented manifolds, where the Calabi-Yau pairing encodes the intersection theory of oriented submanifolds. Since the sequence of Calabi-Yau algebras $\cF_p$ depends on the choice of $\omega$, an answer to our question amounts to looking for an intersection theory tailored to symplectic manifolds---and one not involving holomorphic curve theory.

In Section~\ref{section.intersection}, we begin to interpret the intersection pairing when $\omega$ is an integral cohomology class and $p=0$. The interpretation is as follows: Consider the contact circle bundle $E$ associated to the line bundle classified by $\omega$. Given an isotropic $I$ inside $M$, one can lift $I$ to a section $\tilde I$ of $E$ because $\omega|_I = 0$. As it turns out, isotropics naturally pair with coisotropics $C \subset M$, but $C$ may have some non-trivial, symplectic boundary. (This is a manifestation of the fact that the Poincar\'e dual form of $C$ may not be closed under the de Rham differential.) Regardless, our work suggests that $C$ lifts to a submanifold $\tilde C \subset E$ in the circle bundle which no longer has boundary. Then the intersection of $\tilde I$ with $\tilde E$---hence the usual intersection pairing on $E$---is the Calabi-Yau pairing on $\cF_0\,$.

Second: What formal properties do the $\cF_p$ satisfy? For instance, what kinds of symplectic morphisms induce maps between them?

We prove that the algebras $\cF_p$ enjoy several formal properties in Section~\ref{section.functoriality}. For instance, Theorem~\ref{thm.cone} implies that $\cF_p$---as an $A_\infty$-algebra---is only an invariant of the cohomology class of $\omega$, rather than $\omega$ itself. Second, since there are natural quotient maps  $\cone(\omega^{p+1}) \to \cone(\omega^{p})$, one obtains a sequence of $A_\infty$-algebra maps that go in a reverse direction to the  filtration that defines $\cpF$:
	\eqn\label{eqn.algebra-maps}
		\ldots \to \cF_p \to \cF_{p-1} \to \ldots \to \cF_0.
	\eqnd 
This sequence of algebra maps suggests that the $\cF_p$ can be given a simple interpretation in terms of derived geometry---they encode higher order neighborhoods around the subvariety of $\Omega(M)$ determined by $\omega$. 

Furthermore, we begin to develop the functorial properties of the assignment $M \mapsto \cone(\omega^\bullet)$. For instance, this assignment is functorial with respect to maps $M \to M'$ that respect the symplectic form. Moreover, the assignment $M \mapsto \cF_p(M)$ forms a sheaf of cdgas on the site of symplectic manifolds with open embeddings. Finally, we begin to see the seeds of Weinstein functoriality: For any Lagrangian submanifold $L \subset (M_1 \times M_2, -\omega_1 \oplus \omega_2)$, one can define a cdga $\cF_p(L)$ receiving an algebra map from each of the $\cF_p(M_i)$---in particular, each Lagrangian correspondence defines a bimodule for the algebras $\cF_p(M_i)$. It will be a subject of later work to show that this assignment sends Lagrangian correspondences to tensor products of bimodules.

{\bf Acknowledgments.}
We thank Kevin Costello, Rune Haugseng, Nitu Kitchloo, Si Li, Richard Schoen, Xiang Tang, Chung-Jun Tsai, and Shing-Tung Yau for helpful discussions.  This research is supported in part by National Science Foundation under Award No. DMS-1400761 of the first author and a Simons Collaboration Grant of the second author.  We are also grateful for the hospitality of the Perimeter Institute for Theoretical Physics and Harvard University's Center for Mathematical Sciences and Applications while this research took place. Research at Perimeter Institute is supported by the Government of Canada through the Department of Innovation, Science and Economic Development and by the Province of Ontario through the Ministry of Research and Innovation.

\section{Recollections}
Here we recall the constructions and results from \cite{tseng-yau}, \cite{tseng-yau-2}, \cite{tsai-tseng-yau}. This will allow us to set some notation, and to recall the $A_\infty$-operations for the algebras $\cF_p$. (Here, $p$ can be any integer between $0$ and $n$, inclusive.)

\subsection{Form decomposition}
Fix a $2n$-dimensional symplectic manifold $(M^{2n},\omega)$. Since $\omega$ defines an isomorphism between 1-forms and vector fields, $\omega$ itself is identified with a bivector field (the Poisson bivector field), under this isomorphism.  There are three basic operations on de Rham forms on $M$:
\begin{enumerate}
	\item[(1)]
		$L: \Omega^k \to \Omega^{k+2}$ sends $\eta \mapsto \omega \wedge \eta$.
	\item[(2)] 
		$\Lambda: \Omega^k \to \Omega^{k-2}$ sends $\eta \mapsto \Lambda \eta$, the interior product with the Poisson bivector field.
	\item[(3)]
		$H: \Omega^k \to \Omega^k$ sends a $k$-form $\eta$ to $(n-k)\eta$.
\end{enumerate}
    
Together, $\{L, \Lambda, H\}$ generate an $sl_2$ Lie algebra acting on forms.  
The highest weight forms under this action are called {\em primitive} forms, whose space we shall denote by $P^{\bullet}$.  If $\beta_k \in P^k$ is a primitive $k$-form, then its interior product with the Poisson bivector field vanishes, i.e. $\Lambda\, \beta_k = 0$.  As usual with $sl_2$ representations, any $k$-form $\eta_k\in\Omega^k$  can be expressed as a sum of primitive forms wedged with powers of $\omega$: 
	\eqnn
	\eta_k = \sum_{k=2j+s} L^j \beta_s = \beta_k + L \beta_{k-2} + \ldots + L^p \beta_{k-2p} +  \ldots 
	\eqnnd
This Lefschetz decomposition of $\Omega^\bullet$ can be arranged in a ``pyramid" diagram as shown in Figure \ref{figure.triangle}.  (This can be compared to the well-known $(p,q)$ ``diamond" of forms on complex manifolds.)  Associated with this Lefschetz decomposition, we define four more operations:
\begin{enumerate}
\item[(4)] $L^{-p}$: the negative power of $L$ with $p>0$.  Roughly, it is the ``inverse" of $L^p$ and removes $\om^p$ from each term of the decomposition.  Explicitly, 
	\eqnn
	L^{-p} \eta_k =  \sum_{k=2j+s} L^{-p}L^j \beta_s  =  \beta_{k-2p} + L \beta_{k-2p-2} + \ldots 
	\eqnnd
\item[(5)] $\rstar\,$: the natural reflection action about the central axis of the pyramid diagram in Figure \ref{figure.triangle}.  The operation is explicitly defined on each term of the Lefschetz decomposition by
	\eqn\label{eqn.star}
	\ast_r(L^j \beta_s) = L^{n-j-s}\beta_s\,.
	\eqnd
	It is worthwhile to note that in terms of $\rstar$, 
	\eqnn
	L^{-p} = \rstar \, L^p \, \rstar\,,
	\eqnnd
	that is, $L^{-p}$ is the $\rstar$ conjugate of $L^p$.  Also, acting on $\Om$, we have $\rstar\, \Om^k = \Om^{2n-k}$ as expected.  
\item[(6)] $\pip$: a projection operator that keeps only terms up to the $p$th power of $\omega$ in the Lefschetz decomposition:
\eqn\label{projp}
\pip \eta_k = \sum_{j=0}^{p} L^j \beta_s = \beta_k + \ldots + L^p \beta_{k-2p}
\eqnd

Alternatively, $\pip$ can be defined as
	\eqnn
	\Pi^p = 1 - L^{p+1} \circ L^{-(p+1)}.
	\eqnnd

\item[(7)] $\pips$: the $\rstar$ conjugate of $\pip$ defined by
	\eqnn
		\pips = \rstar \pip \rstar= 1 - L^{-(p+1)} \circ L^{p+1}.
	\eqnnd

\end{enumerate}

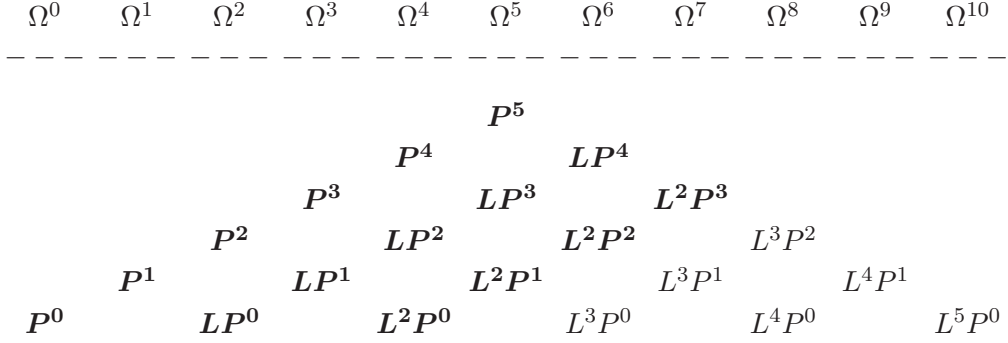
\begin{figure}[t]
$$\xymatrix
@R=0pt
@C=-2pt
@H=2pt
@!C
{
\Omega^0 & \Omega^1 & \Omega^2 & \Omega^3 & \Omega^4 & \Omega^5 & \Omega^6 & \Omega^7 & \Omega^8 & \Omega^9 & \Omega^{10}
\\
---&---&--- &---&---&---&---&---&---&---&---\\
\\
 & & & & & \bm{P^5} & & & & &\\
 & & & & \bm{P^4}& &\bm{L P^4} & & & &\\
 & & &\bm{P^3}& &\bm{L P^3} & & \bm{L^2 P^3} &&&  \\
 & &\bm{P^2}& & \bm{L P^2}& &\bm{L^2 P^2} & &L^3 P^{2} & &\\
 & \bm{P^1} & & \bm{L P^1} & &\bm{L^2 P^1}& &L^3P^{1} & &L^4 P^{1} &\\
\bm{P^0}& &\bm{L P^0}& &\bm{L^2 P^0}& &L^3P^{0}& &L^4P^{0} & & L^5 P^{0} \\
}$$
\caption{The decomposition of  forms into a ``pyramid" diagram in dimension $2n=10$.   Note that there is a natural reflection symmetry action about the $\Omega^{n=5}$ column. This is the $\rstar$ action defined in (\ref{eqn.star}). To illustrate filtered forms, we have {\bf bold-faced} those elements  in $F^p\Omega$ for $p=2$ here.}
\label{figure.triangle}
\end{figure}

The decomposition by powers of $\omega$ allows us to group together forms that contain terms only up to certain powers of $\omega$.

\begin{defn}\label{defn.fforms}
Define the space of {\em $p$-filtered $k$-forms} to be
	\eqnn
	\FO{k} := \pip \Omega^{k}
	\eqnnd
for $p\in \{0, 1, \ldots, n\}$.  
\end{defn}

The filtered $F^p\Omega^k$ spaces give us a natural filtration of $\Omega^k$:
\begin{align*}
P^k := F^{0}\Omega^{k} \subset F^{1}\Omega^{k} \subset F^{2}\Omega^{k} \subset \ldots \subset F^{n}\Omega^{k} = \Omega^k.
\end{align*}
Note that the 0-filtered forms are precisely primitive forms, and that $F^n\Omega^k = \Omega^k$.  
\begin{lemma}\label{lemma.filtered}
The following are equivalent conditions for $p$-filtered forms:
\begin{enumerate}
\item[(i)] $\alpha_k \in \FO{k}\,$;
\item[(ii)] $\Lambda^{p+1} \alpha_k = 0$  ({or equivalently}, $L^{-(p+1)} \alpha_k =0)\,$;
\item[(iii)] $L^{n+p+1-k} \alpha_k = 0\,$;
\item[(iv)] $  L^{p+1} \rstar \alpha_k = 0\,$.
\end{enumerate}
\end{lemma}
\begin{remark}
The $\FO{k}$ space is only non-zero for $k=0, 1, \ldots, (n+p)\,.$  See  Figure~\ref{figure.triangle}.
 \end{remark}

\subsection{The cochain complex \texorpdfstring{$\cF^\bullet_p$}{Fp}}\label{section.cA}
There are natural differential operators $(\ddp, \ddm)$ that maps between $p$-filtered forms.  We recall the key ingredients here:
	\enum
				\item 
	On symplectic manifolds, the exterior differential has a natural decomposition
	\eqnn
	d = \dpp + \om \w \dpm\,.
	\eqnnd
	where $\dpp: L^j P^s \to L^j P^{s+1}$ and $\dpm: L^j P^s \to L^j P^{s-1}$.  Note in particular when $j=0$, $\dpp$ and $\dpm$ map primitive forms into primitive forms.  The pair of operators $(\dpp, \dpm)$ square to zero ($\dpp^2 = \dpm^2 = 0$), commute with $L$ ($ [L, \dpp] = [L, L\dpm] =0$), and also graded commute after applying $L$: $L \dpp \dpm = - L \dpm \dpp$.  
	\item On $p$-filtered forms, we can define two natural linear operators:
\begin{align} 
\ddp &= \pip\, \circ\, d\,,  \label{ddpdef}\\	
\ddm & = \rstar\, d\; \rstar\,,\label{ddmdef}
\end{align}
such that $\ddp:\FO{k} \to \FO{k+1}$ and $\ddm: \FO{k} \to \FO{k-1}$.  Both operators square to zero ($(\ddp)^2 = (\ddm)^2 =0 $).  In fact, when acting on primitive forms, $d_{\pm} P^k = \partial_{\pm} P^k$.  Furthermore, ($\ddp, \ddm$) is related to $(\dpp, \dpm)$ by
\begin{align*}
\ddm &= \dpm + \dpp L^{-1}\\
(\dpp\dpm) \ddp&= \ddm (\dpp\dpm) = 0
\end{align*}
	\enumd

The differentials $(\ddp, \ddm, \dpp\dpm)$ can be used to write down an elliptic complex for $\FO{\bullet}$:
\begin{theorem}\label{ellipticC}[Section 3 of~\cite{tsai-tseng-yau}]
The differential complex
\eqnn
0\xra{} \FO{0}  \xra{\ddp}  \FO{1} \xra{\ddp} \ldots 
	\xra{\ddp} 	\FO{n+p} \xra{\dpp \dpm}
	\FO{n+p}\xra{\ddm}  \FO{n+p-1}\xra{\ddm} \ldots\xra{\ddm}\FO{0} \xra{} 0
\eqnnd
is elliptic for all $p\in\{0, 1, \ldots, n\}$. 
\end{theorem}
As a corollary of the ellipticity of the complex, the associated filtered cohomologies labelled by $$F^pH=\{F^pH_+^0, \ldots, F^pH_+^{n+p}, F^pH_-^{n+p}, \ldots, F^pH_-^0\}$$
are all finite-dimensional. 

The cochain complex $\cF_p$ is defined following the elliptic complex:   
\begin{align}\label{eqn.A}
	\cpF = 
		\left(\FO{0}  \xra{\ddp} \ldots
	\xra{\ddp} 	\FO{n+p} \xra{-\dpp \dpm}
	\BFO{n+p}\xra{-\ddm} \BFO{n+p-1}
	\xra{-\ddm}\ldots\xra{-\ddm}\BFO{0}\right). 
\end{align}
As vector spaces $\BFO{k} = \FO{k}$.  The extra bar is only inserted to distinguish the second set of $\FO{\bullet}$ in the complex from the first.  Also, the additional minus signs in some of the differentials are needed to satisfy the $A_\infty$ relations described below.  At times, we will just denote the above cochain complex as $\cF^\bullet$, with
	\eqnn
	\cF^j =
	\begin{cases}
	\FO{j} & j \leq n \\
	\BFO{2n+2p+1-j} &  j \geq n+p+1 \\
	0 & j < 0, ~j>2n+2p+ 1.
	\end{cases}
	\eqnnd
The notation here is that $\BFO{k}$ has degree $2n+2p+1-k$ in this complex, while $\FO{k}$ has degree $k$.  If ambiguity may arise whether a $p$-filtered $k$-form is an element of $\cF^k$ or $\cF^{2n+2p+1-k}$, we will add a bar and write $\ov \alpha_k \in \BFO{k}$ to denote an element in $\cF^{2n+2p+1-k}$.

\subsection{The \texorpdfstring{$A_\infty$}{A-infinity}-algebra structure}\label{section.AooF}
We use the sign convention that a series of operations
	\eqnn
	m^l : \cF^{\,\tensor\, l} \to \cF[2-l]
	\eqnnd
satisfies the $A_\infty$ relations if and only if
	\eqnn
	\sum (-1)^{r + st} m^{r + t + 1} (\id^{\,\tensor\, r} \tensor\;  m^s \tensor \id^{\,\tensor\, t}) = 0\,.
	\eqnnd
As an example, if $m^l=0$ for $l \geq 3$, then the data $(\cF, (m^l))$ define an ordinary differential graded algebra (dga). 

{\bf The $m^2$ operation.}
\enum
	\item
	 	For $a_1 \in \FO{k_1}=\cF^{k_1}\,$, $a_2 \in \FO{k_2}=\cF^{k_2}\,$, and  $k_1+k_2 \leq n+p\,$, we set
			\eqnn
			m^2(a_1,a_2) = \Pi^p(a_1 \wedge a_2)
			\eqnnd
		where $\Pi^p$ is the projection to the $p$-filtered components defined above. 
	\item For $a_1 \in \FO{k_1}=\cF^{k_1}\,$, $a_2 \in \FO{k_2}=\cF^{k_2}\,$, and $k_1+k_2 > n+p\,$, define
			\eqnn
			D_L(a_1,a_2) = -d\,L^{-(p+1)}\!(a_1 \wedge a_2) + (L^{-(p+1)}da_1) \wedge a_2 + (-1)^j a_1 \wedge (L^{-(p+1)}da_2).
			\eqnnd
			Then, 
			\eqnn
			m^2(a_1,a_2)
			=
			\Pi^p \ast_r ( D_L (a_1,a_2) ).
			\eqnnd
		In other words, the $m^2$ operation measures the difference between the action of $d\,L^{-(p+1)}$ and the derivation of $L^{-(p+1)} d$ on the wedge product of filtered forms. (It then picks out the $p$-filtered part of the $\ast_r$ reflection.)
	\item
		For  $a_1 \in \FO{k_1}=\cF^{k_1}$ and $a_2\in \BFO{k_2} = \cF^{2n+2p+1-k_2}$, we set
			\eqnn
			m^2(a_1, a_2)
			=
			(-1)^{k_1} \ast_r (a_1 \wedge (\ast_r  a_2) ) .
			\eqnnd
	\item 
		For $a_1\in \BFO{k_1}= \cF^{2n+2p+1-k_1}$ and $a_2 \in \FO{k_2}=\cF^{k_2}\,$, we set
			\eqnn
			m^2(a_1, a_2) = \ast_r \left( (\ast_r  a_1) \wedge a_2 \right).
			\eqnnd
		This implies that $m^2$ is graded commutative:
	\eqnn
	m^2(a_1, a_2) = (-1)^{a_1a_2}m^2(a_2,a_1).
	\eqnnd

	\item
		And for degree reasons, if $a_1\in \BFO{k_1}$ and $a_2\in \BFO{k_2}$, then 
			\eqnn
			m^2( a_1, a_2) = 0\,.
			\eqnnd
\enumd

{\bf The $m^3$ operation.} Let $a_i \in \cF^{k_i}$.  We set
	\eqnn
	m^3(a_1,a_2,a_3)
	=
		\begin{cases}
			\Pi^p \ast_r \big[a_1 \wedge L^{-(p+1)} (a_2 \wedge a_3)  & k_1+k_2+k_3 \geq n+p+2\,,  \\  
\qquad\qquad\qquad  - L^{-(p+1)} (a_1 \wedge a_2) \wedge a_3\big]	& \text{~~~with~} k_1, k_2, k_3 \leq n+p \\

			0 & \text{otherwise}
		\end{cases}
	\eqnnd

{\bf The higher $m^l$ operations.} We set $m^l = 0$ for all $l \geq 4$.

From \cite{tsai-tseng-yau}, we have the following:

\begin{theorem}[Section 5 of~\cite{tsai-tseng-yau}]
For any symplectic manifold $M$, the operations $m^l$ endow $\cF(M)$ with the structure of an $A_\infty$-algebra.
\end{theorem}

\section{Equivalence of algebras}

\subsection{Definition of \texorpdfstring{$\cC$}{C}}
Let $f: X \to Y$ be a map of cochain complexes. In homological algebra, the natural notion of quotient is the {\em cone}, or {\em mapping cone} of $f$. It is defined to be the cochain complex
	\eqnn
		Y \oplus X[1],
		\qquad
		d(y\oplus x) = dy + f(x) \oplus -dx.
	\eqnnd
Finally, a {\em commutative differential graded algebra}, or {\em cdga}, is a commutative algebra in the category of cochain complexes. Explicitly, a cdga is a cochain complex $A = (A,d)$ together with a map of cochain complexes
	\eqnn
		m^2: A \tensor A \to A
	\eqnnd
which is associative, and which satisfies graded commutativity:
	\eqnn
			m^2(a,b) = (-1)^{|a| \, |b|} m^2(b,a).
	\eqnnd
Note that $m^2$ is a map of cochain complexes if and only if it satisfies the Leibniz rule:
	\eqnn
		dm^2(a,b) = m^2(da,b) + (-1)^{|a|} m^2(a,db).
	\eqnnd
A map of cdgas is a map of cochain complexes $f: A \to B$ such that $f \circ m^2 = m^2 \circ f \tensor f$. 

\begin{remark}
Given any commutative ring $R$, and an element $x$, one has the map of $R$-modules 
	\eqnn
	R \xra{x} R
	\eqnnd
and one can define the quotient $R$-module $R/xR$; this also has a commutative ring structure. In what follows, we consider the case of $R = \Omega^\bullet(M)$, the de Rham complex of differential forms on a symplectic manifold $M$. Our element $x$ is some power of $\omega$.\footnote{In fact, the discussion holds for any even-degree, closed differential form on $M$.} So, heuristically, $\cC$ should be thought of as ``$R/xR$.''
\end{remark}

Consider the degree 0 map
	\eqnn
		\omega^{p+1}: \Omega(M)[-2p-2] \xra{\w\, \omega^{p+1}} \Omega(M).
	\eqnnd
Then $\cone(\omega^p)$ has in degree $k$ the vector space  
	\eqnn
	\Omega^k(M) \oplus \theta\,\Omega^{k-2p-1}(M)
	\eqnnd
where $\theta$ is a formal element of degree $2p+1$. This has differential
	\eqnn
	d_{\cC}(\eta \oplus \theta \xi)
	=
	(d\eta + \omega^{p+1} \xi )\oplus
	-\,\theta\, d\xi.
	\eqnnd
This cochain complex has another description endowing it with an obvious cdga structure:

\begin{defn}\label{defn.C}
We define $\cC$ to be the cdga obtained from $\Omega(M)$ by freely attaching a degree $2p+1$ variable $\theta$ satisfying $d\theta - \omega^{p+1} = 0$. Note that $\omega$ and $p$ are implicit in writing $\cC$.
\end{defn}

Thinking of the cdga  $\cC$ as an $A_\infty$-algebra, we will let $m^1_\cC$ and $m^2_\cC$ denote its differential and its multiplication, respectively, while $m^k_\cC = 0$ for all $k \geq 3$. 
Explicitly, we have
	\eqnn
	m^2_\cC
	\left((\alpha + \theta \alpha') \tensor (\beta + \theta \beta')\right) 
	=
	(\alpha + \theta \alpha') \wedge (\beta + \theta \beta')
	=
	\alpha\wedge \beta \oplus \theta(\alpha' \wedge \beta + (-1)^{\alpha}\alpha \wedge \beta').
	\eqnnd

\subsection{\texorpdfstring{$\cC$}{C} is equivalent to \texorpdfstring{$\cF$}{F}}
Throughout this section we fix an integer $0 \leq p \leq n$. The dependence of $\cC$ and $\cF$ on this choice is implicit.   Furthermore, the indices $j$ and $k$ take integral values $j\in\{0,1,\ldots, 2n+2p+1\}$ and  $k\in\{0, 1, \ldots, n+p\}$.   

We shall write any $p$-filtered form $\alpha \in \FO{k}$ as
\begin{align*}
\alpha_k = \beta_k + \om   \beta_{k-2} + \ldots + \om^p   \beta_{k-2p} 
\end{align*}
where each $\beta_i\in P^i$ is a primitive form of degree $i$.  

Let us first make the observation that
\begin{align*}
	\cC^k &= \Omega^k \oplus \,\theta\,\Omega^{k-2p-1}\\
	\cC^{2n+2p+1-k} &= \Omega^{2n+2p+1-k} \oplus\, \theta\,\Omega^{2n-k} = \left(\rstar\, \Omega^{k-2p-1}\right) \oplus \, \theta \, \left(\rstar \Omega^k\right)  
\end{align*}
making use of the relation $\rstar \Om^k(M) = \Om^{2n-k}(M)\,$ for any integer $k$.
The above implies $\cC^k \cong \cC^{2n+2p+1-k}$ as expected.  Following  this, we can write for any $C_k \in \cC^k$
\begin{align}\label{Ckexp} 
C_k = (\alpha_k + \om^{p+1}   \beta_{k-2p-2} + \ldots ) + \thp   ( \alpha_{k-2p-1} + \om^{p+1}   \beta_{k-4p-3} + \ldots ) \,,
\end{align}
and for any $C_{2n+2p+1-k}\in \cC^{2n+2p+1-k}$
\begin{align}\label{Cnkexp}
C_{2n+2p+1-k}  = \rstar ({\alpha}_{k-2p-1} + \om^{p+1}   { \beta}_{k-4p-3} + \ldots ) + \thp  \ast_r ({\alpha}_k + \om^{p+1}   { \beta}_{k-2p-2} + \ldots)\,.
\end{align} 
If we further replace the $\alpha_{k-2p-1}$ terms in \eqref{Ckexp}-\eqref{Cnkexp} above with its Lefschetz decomposed form
$$\alpha_{k-2p-1} = \beta_{k-2p-1} + \om \beta_{k-2p-3} +\ldots + \om^p \beta_{k-2p}\,,$$ 
then we find the relation
\begin{align*}
\cC^k \cong \cC^{2n+2p+1-k} \cong \left(\bigoplus_{0 \leq j \leq k-2p-1} P^j \right)\; \oplus\;\;\FO{k} 
\end{align*}
Therefore, the maps 
\begin{align}
	C_k &\mapsto (\beta_0,\ldots,\beta_{k-2p-1}, \alpha_k)\label{cbd1}\\
	\qquad
	C_{2n+2p+1-k}& \mapsto (\beta_0, \ldots, \beta_{k-2p-1}, \alpha_k)\label{cbd2}
\end{align}
are isomorphisms of $\RR$ vector spaces.  This gives another representation of elements in $\cC^j$.

\begin{figure}[t]
\begin{align*}
\xymatrix
{
 \ldots \ar[r]^-d & \cC^k \ar[d]^-{\alpha_k}\ar[r]^-d& \ldots\ar[r]^-d& \cC^{n+p} \ar[d]^-{\alpha_{n+p}} \ar[r]^-d& \cC^{n+p+1}\ar[d]^-{ -({\alpha}_{n+p} +  w^p \dpp{\beta}_{n-p-1})} \ar[r]^-d& \ldots \ar[r]^-{d} & \cC^{2n+2p+1-k}\ar[d]^-{ -({\alpha}_k +  w^p \dpp{\beta}_{k-2p-1})} \ar[r]^-{d} & \ldots \\
 \ldots \ar[r]^-{\ddp} & \FO{k} \ar[r]^-{\ddp}& \ldots \ar[r]^-{\ddp} & \FO{n+p} \ar[r]^-{-\dpp\dpm} & \BFO{n+p} \ar[r]^-{-\ddm} &\ldots  \ar[r]^-{-\ddm} & \BFO{k} \ar[r]^-{-\ddm}& \ldots
}
\end{align*}
\caption{The map $f:\cC \to \cF$.  The maps are expressed in terms of the decomposition in \eqref{Ckexp}-\eqref{Cnkexp}.}
\label{figure.mapf}
\end{figure}
\begin{figure}[t]
\begin{align*}
\xymatrix
{
 \ldots \ar[r]^-d & \cC^k \ar[r]^-d& \ldots\ar[r]^-d& \cC^{n+p}  \ar[r]^-d& \cC^{n+p+1} \ar[r]^-d& \ldots \ar[r]^-{d} & \cC^{2n+2p+1-k} \ar[r]^-{d} & \ldots \\
 \ldots \ar[r]^-{\ddp} & \FO{k} \ar[u]^-{\alpha_k-\theta \dpm\beta_{k-2p}}\ar[r]^-{\ddp}& \ldots \ar[r]^-{\ddp} & \FO{n+p} \ar[u]^-{\om^p\beta_{n-p} - \theta \dpm\beta_{n-p}}\ar[r]^-{-\dpp\dpm} & \BFO{n+p} \ar[u]^-{-\theta \beta_{n-p}}\ar[r]^-{-\ddm} &\ldots  \ar[r]^-{-\ddm} & \BFO{k} \ar[u]^-{-\theta \rstar \alpha_k} \ar[r]^-{-\ddm}& \ldots
}
\end{align*}
\caption{The map $g: \cF \to \cC$.  In the maps above, we use the notation $\beta_{k-2p} = L^{-p}\alpha_k\,$,  and for the two maps in the middle, we have applied the relation $\FO{n+p}=\BFO{n+p} = L^p P^{n-p}$.}
\label{figure.mapg}
\end{figure}

We now define a pair of maps $(f, g)$ that relates $\cC$ and $\cF$. 
\begin{defn}\label{defn.fg}
In terms of \eqref{Ckexp}-\eqref{Cnkexp},  we define  
\begin{align*}
f:~~&~~ \cC^j &&\to \qquad\quad \cF^j\\
&~~ C_j &&\mapsto 
\begin{cases}
\qquad\, \alpha_j  & j \leq n+p \,, \\
-({\alpha}_k +  \pips w^p d{\alpha}_{k-2p-1})  & j\geq n+p+1,\, j=2n+2p+1-k \,,
\end{cases}
\end{align*}
where $\pips w^p d{\alpha}_{k-2p-1} = \om^p\dpp{\beta}_{k-2p-1}\,$, and 
\begin{align*}
g:~~&~~\cF^j &&\to \qquad\qquad \cC^j\\
&~~\alpha_k &&\mapsto 
\begin{cases}\alpha_k - \thp  L^{-(p+1)} d \alpha_k  & ~~ j\leq n+p,~k=j\,,\\
- \,\thp \ast_r {\alpha}_k &~~ j\geq n+p+1,~k=2n+2p+1-j\,,
\end{cases}
\end{align*}
where $L^{-(p+1)}d\alpha_k= \dpm\beta_{k-2p}\,$. 
\end{defn}
A pictorial description of the $f$ and $g$ maps is given in Figure \ref{figure.mapf} and Figure \ref{figure.mapg}, respectively.  Alternatively, we can also express the maps in terms of the decomposition in \eqref{cbd1}-\eqref{cbd2}.  We can write  $f: \cC^j \to \cF^j$ as 
\begin{align}\label{fdecomp}
f:  (\beta_0,\ldots,\beta_{k-2p-1}, \alpha_k)  \mapsto
\begin{cases}  
\alpha_k &    j\leq n+p,~ k=j\\
-({\alpha}_k +  w^p \dpp{\beta}_{k-2p-1} ) & j\geq n+p+1,~ k=2n+2p+1-j
\end{cases}
\end{align}
and $g: \cF^j \to \cC^j$ as
\begin{align}\label{gdecomp}
g&:  \alpha_k  \mapsto 
\begin{cases} 
(0,\ldots, - L^{-(p+1)} d \alpha_k, \alpha_k) & j\leq n+p,~ k=j\\
(0,\ldots, 0, - {\alpha}_k)& j\geq n+p+1,~ k=2n+2p+1-j
\end{cases}
\end{align}

\begin{lemma}\label{lemma.chain}
The maps $f: \cC \to \cF$ and $g: \cF \to \cC$ are chain maps, i.e. 
\begin{align*} d_{\cF}\, f = f\, d_{\cC}\,, \qquad  d_{\cC}\, g = g\, d_{\cF}\,.
\end{align*}
\end{lemma}
\begin{proof}
The lemma is proved by direct calculation.  We first compute
\begin{align}\label{chainc1}
		d_{\cF}\, f\, C_j = 
			\begin{cases} 
			 \ddp \alpha_j 
				& j < n+p \\
			- \om^p   \dpp\dpm \beta_{j-2p} 
				& j=n+p \\
			\ddm \alpha_k  - \om^p   \dpp\dpm \beta_{k-2p-1}   
			    & j=2n+2p+1-k > n+p , \quad k \leq n+p
			\end{cases}
\end{align}
To compute $f d_{\cC}$, we consider first the case $j< n+p$.  We have
\begin{align*}
d_{\cC} C_j & = (\ddp \alpha_j  +  \om^{p+1}  (\dpm \beta_{j-2p} + \dpp \beta_{j-2p-2}) + \ldots) + \om^{p+1}   (\alpha_{j-2p-1} + \om^{p+1}   \beta_{j+4p -3} + \ldots) \\ &\qquad\qquad - \thp   d(\alpha_{j-2p-1} + \om^{p+1}   \beta_{j+4p -3} + \ldots) \nonumber
\end{align*}
It then follows that $f\, d_{\cC}\, C_j = \ddp \alpha_j$.  When $j=n+p$, the above equation becomes
\begin{align*} 
d_{\cC} C_{n+p} & = \om^{p+1}   (\dpm \beta_{n-p} + \dpp \beta_{n-p-2} + \beta_{n-p-1}) + \cO(\om) \ldots) \\
& \qquad  -  \thp   d(\alpha_{n-p-1} + \om^{p+1}   \beta_{n+3p -3} + \ldots) 
\end{align*}
Notice that $\thp   d \alpha_{n-p-1} = \thp   (\dpp \beta_{n-p-1} + \cO(\omega) ) = \thp \ast_r (\om^p   \dpp \beta_{n-p-1} + \cO(\om^{p+1})).$  Hence, we obtain
\begin{align*}
f\, d_{\cC}\, C_{n+p} = - [-\om^p   \dpp \beta_{n-p-1} + \om^p   \dpp(\dpm \beta_{n-p} + \beta_{n-p-1})] = - \om^p   \dpp\dpm \beta_{n-p}~.
\end{align*}
Now for $j>n+p$, let $j=2n+2p+1-k$ where $0\leq k\leq n+p$.  Then 
\begin{align*}
d_{\cC} C_{j} &= \om^{n+2p+2-k} ( \dpm \beta_{k-2p-1}+\dpp \beta_{k-2p-3} + \beta_{k-2p-2}  + \cO(\om) ) \\ & \quad - \thp L^{n-k+1} \left[(\dpm \beta_k + \dpp \beta_{k-2}) + \ldots- \om^p (\dpm \beta_{k-2p} + \dpp \beta_{k-2p-2}) + \dots\right] 
\end{align*}
This gives the desired result 
\begin{align*}
f \, d_{\cC}\, C_{2n+2p+1-k} & = - (- \ddm \alpha_k - \om^p   \dpp \beta_{k-2p-2} + \om^p   \dpp ( \dpm \beta_{k-2p-1} +  \beta_{k-2p-2}) \\ 
& = \ddm \alpha_k - \om^p   \dpp\dpm \beta_{k-2p-1}
\end{align*}
The calculation for the $g$ map is straightforward. It can be easily checked that for $\alpha_k\in \cF^k$, 
\begin{align*}
d_{\cC}\, g (\alpha_k)  = \ddp \alpha_k +  \thp   \dpp\dpm \beta_{k-2p} = g \,d_{\cF} (\alpha_k) \,,
\end{align*}
and for $\alpha_k \in \cF^{2n+2p+1-k}$,
\begin{align*}
d_{\cC}\, g (\alpha_k)& = \thp \rstar \, \ddm \bar{\alpha}_k = g\, d_{\cF} (\bar{\alpha}_k) 
\end{align*}
\end{proof}

It is straightforward to check that  $f \,g = \id_{\cF}$.   Hence, we may wonder whether $g \, f$ admits a chain homotopy to the identity on $\cC$.

\begin{defn}
We define $G: \cC^j \to \cC^{j-1}$ as follows. Fix $\eta \in \Omega^j$ and $\xi \in \Omega^{j-2p-1}$. 
\begin{align*}
G (\eta + \theta \wedge \xi) := L^p\xi + \theta  L^{-(p+1)} \eta.
\end{align*}
\end{defn} 
In terms of the decomposition \eqref{cbd1}-\eqref{cbd2}, $G: \cC^j \to \cC^{j-1}$ acts by sending the element $(\beta_0,\ldots, \beta_{k-2p-1}, \alpha_k)$ to 
\begin{align*}
\begin{cases} 	
(\beta_0,\ldots,\beta_{k-2p-2}, \om^p \beta_{k-2p-1}) 
& j\leq n+p,~ k=j\\
(\beta_0,\ldots, \beta_{k-2p-1}, \alpha_{k})
	& j=n+p+1,~ k=n+p \\
(\beta_0,\ldots,\beta_{k-2p-1}, L^{-p}\alpha_k, 0 ) & j > n+p+1, ~k = 2n+2p+1-j \end{cases} 
\end{align*}
where by convention $\beta_{-1} = 0$.  

Making use of $G$, we have the following result:
\begin{lemma}\label{lemma.retractp}
The inclusion $g: \cF \to \cC$ and the maps $f$ and $G$ exhibit $\cF$ as a strong deformation retract of $\cC$:
\eqnn
	f \,g = \id_\cF,
	\qquad
	\id_{\cC}  - g\, f = d_{\cC} G + G d_\cC.
	\eqnnd
In particular, both $f$ and $g$ are quasi-isomorphisms and $H^*(\cF) \cong H^*(\cC)$.
\end{lemma}

\begin{proof}[Proof of Lemma~\ref{lemma.retractp}.]
As mentioned, the proof of the first relation $f g = \id_\cF$ follows directly from the definition of the maps.  For the second relation, consider first for $\cC^{j=k}$ with $0\leq k \leq n+p$.  
For the lefthand side, we have
\begin{align*}
(g f) C_k &= \alpha_k - \thp   \dpm \beta_{k-2p} \\
(\id_{\cC}  - g f) C_k &=  (\om^{p+1} \beta_{k-2p-2}+ \ldots ) + \thp   ( \dpm \beta_{k-2p} + \alpha_{k-2p-1} + \om^{p+1}   \beta_{k-4p-3} + \ldots)
\end{align*}
For the righthand side, we have
\begin{align*}
G C_k & = \om^p (\alpha_{k-2p-1} + \om^{p+1} \beta_{k-4p-3} + \ldots) + \thp ( \beta_{k-2p-2} + \om  \beta_{k-2p-4} + \ldots )\\
d_{\cC} G C_k &  = \om^p   d(\alpha_{k-2p-1} + \om^{p+1} \beta_{k-4p-3} + \ldots) +\om^{p+1} (\beta_{k-2p-2} + \om  \beta_{k-2p-4} + \ldots ) \\  & \qquad\qquad - \thp   d(\beta_{k-2p-2}  + \ldots )\\
G d_{\cC} C_k & = - \om^p   d(\alpha_{k-2p-1} + \om^{p+1}   \beta_{k+4p -3} + \ldots)\\ &\quad +\thp   \left[ \dpm \beta_{k-2p} + d(\beta_{k-2p-2} +   \ldots ) + (\alpha_{k-2p-1} + \om^{p+1}   \beta_{k+4p -3} + \ldots)\right]
\end{align*}
which results in
\begin{align*}
(d_{\cC} G + G d_\cC) C_k & = \om^{p+1} (\beta_{k-2p-2} + \om  \beta_{k-2p-4} + \ldots ) \\ & \qquad + \thp   (\dpm \beta_{k-2p} + \alpha_{k-2p-1} + \om^{p+1}   \beta_{k+4p -3} + \ldots)
\end{align*}
which matches exactly with the lefthand side.

We now check the case for $\cC^{j=2n+2p+1-k}$ with $0\leq k \leq n+p$.  For the lefthand side, we have
\begin{align*}
g f C_{2n+2p+1-k} & = \thp \ast_r (\alpha_k + \om^p \dpp \beta_{k-2p-1})\\
(\id_{\cC}  - g f) C_{2n+2p+1-k} &= \om^{n+2p+1-k} (\alpha_{k-2p-1} + \om^{p+1}   \beta_{k-4p-3} + \ldots) \\ & \quad + \thp \ast_r (-\om^p \dpp \beta_{k-2p-1} + \om^{p+1}    \beta_{k-2p-2} +  \ldots) 
\end{align*}
For the righthand side, we have
\begin{align*}
GC_{2n+2p+1-k} 
	& = \om^{n-k+p}   (\alpha_k + \om^{p+1}  \beta_{k-2p-2} + \ldots)  \\
	&\quad + \thp   \om^{n-p+k}  (\alpha_{k-2p-1} + \om^{p+1}  \beta_{k-4p-3}+ \ldots)\\
d_{\cC} G\, C_{2n+2p+1-k} 
	& =  \om^{n-k+p}    d(\alpha_k + \om^{p+1}  \beta_{k-2p-2} + \ldots) \\ 
	& \quad + \om^{n-k+2p+1}  (\alpha_{k-2p-1} + \om^{p+1}  \beta_{k-4p-3} + \dots)\\ 
	& \quad 
- \thp   \om^{n+p-k}   d(\alpha_{k-2p-1} + \om^{p+1}  \beta_{k-4p-3} + \ldots)\\
d_{\cC} C_{2n+2p+1-k} 
	& = \om^{n+2p+1-k}   \left[d(\alpha_{k-2p-1} + \om^{p+1}  \beta_{k-4p-3} + \ldots) - \dpp \beta_{k-2p-1}\right]\\  
	& \quad + \om^{n-k+p+1}  (\om^{p+1}   \beta_{k-2p-2} + \ldots) \\
	& \quad - \thp   \om^{n-k}  d(\alpha_k + \om^{p+1}   \beta_{k-2p-2} +\ldots)\\
G d_{\cC} \,C_{2n+2p+1-k} 
	& = - \om^{n-k+p}  d(\alpha_k + \om^{p+1}   \beta_{k-2p-2} +\ldots)+ \thp    \om^{n-k+p+1}(\beta_{k-2p-2} + \ldots) \\  
	&\quad + \thp    \om^{n+p-k}  \left[d(\alpha_{k-2p-1} + \om^{p+1}  \beta_{k-4p-3} + \ldots)- \dpp \beta_{k-2p-1}\right]
\end{align*}
Note that in the righthand side of the third equation, we have used the property that $L^{p+1} \ast_r \alpha_k= 0$ and also added the term $ - \om^{n+2p+1-k} \dpp \beta_{k-2p-1}$ since it is identically zero.
The above equations give the result
\begin{align*}
d_{\cC} G + G d_{\cC} & = \om^{n-k+2p+1}  (\alpha_{k-2p-1} + \om^{p+1}  \beta_{k-4p-3} + \dots) \\ & \quad  + \thp   \om^{n-k}   \left[-\om^p    \dpp \beta_{k-2p-1}  + (\om^{p+1}   \beta_{k-2p-2} + \ldots) \right]
\end{align*}
which is exactly the lefthand side.
\end{proof}
Though $f$ and $g$ are chain maps, they are clearly not compatible with the product structure of the algebras $\cC$ and $\cF$.  For instance, it is not difficult to see that for generic $a_1, a_2 \in \cF\,$, $g\, m^2_\cF(a_1 \tensor\, a_2) \neq m^2_\cC(ga_1 \tensor \,ga_2)$. So we treat the cdga $\cC$ as an $A_\infty$-algebra.  Recall that a sequence of maps $g^l : \cF^{\tensor l} \to \cC[1-l]$ with $l=1, 2, \ldots\,$ is called an $A_\infty$ map if the following equations are satisfied (see for example Section~3.4 of~\cite{keller-intro}):
\begin{align}\label{Ainmap}
	\sum_{r + s + t = n} (-1)^{r + st}\,g^l (1^{\tensor r} \tensor m^s_\cF \tensor 1^{\tensor t})
	=
	\sum_{i_1 + \ldots + i_q = n} (-1)^{u}\,m^q_\cC(g^{i_1} \tensor \ldots g^{i_q}), \quad n \geq 1
\end{align}
where $l=r+1+t$ and the sign on the right-hand side is given by
	\eqnn
	u = (q - 1)(i_1 - 1) + (q-2) (i_2 - 1) + \ldots + 2(i_{q-2}-1) + (i_{q-1}-1).
	\eqnnd
\begin{theorem}\label{thm.Aoo-map}
The sequence of maps $g^l : \cF^{\tensor l} \to \cC[1-l]$ given by 
\enum
		\item $g^1 = g$,
		\item $
			g^2 = -\thp L^{-(p+1)} m^2_{\cC}(g^1 \tensor g^1),$
			
		\item $g^l = 0$ for all $l \geq 3$.
	\enumd
defines an $A_\infty$-algebra map $\cF \to \cC$.
\end{theorem}
We will give the proof of this theorem in the next subsection.  Since the map $g$ is an quasi-isomorphism, Theorem~\ref{thm.Aoo-map} gives as an immediate corollary  Theorem \ref{thm.cone}, that is, $\cpF  \simeq \cone(\omega^{p+1})\simeq \Omega(E_p)$ as $A_\infty$-algebras.

To simplify notation in sections following, we will also refer the family of $A_\infty$ maps as $g$ as well.

\subsection{Proof of Theorem~\ref{thm.Aoo-map}}

For the proof, we will make use of a few useful formulas:

\begin{lemma}\label{lemma.formulas}
Let $\eta,\xi$ be forms on $M$ with $|\xi| = |\eta| - (2p+1)$. The following hold:
\enum
	\item\label{item.[m,L]}
			\eqnn
				\{m^1_\cC, \theta L^{-(p+1)}\} (\eta + \theta \xi)
				=
				\omega^{p+1} L^{-(p+1)} \eta - \theta
				\left([d,L^{-(p+1)}] \eta - L^{-(p+1)} \omega^{p+1} \xi\right).
			\eqnnd
	\item\label{item.[d,L]}
			\eqnn
				[d,L^{-(p+1)}] \eta = 
				\begin{cases}
					- L^{(p+1)} d\, \pip \eta& |\eta| \leq n + p \\
					 \pips d L^{-(p+1)} \eta & |\eta| > n +p 
				\end{cases}
			\eqnnd
	\item\label{item.g2ab}
		The $g^2$ operation on any $a_1, a_2\in \cF$ can be expressed as
    	\eqnn
      	g^2(a_1,a_2)
      		= 
    			\begin{cases}
    			-\theta L^{-(p+1)}(a_1\wedge a_2) & |a_1|,|a_2| < n + p \\
    			0 & \text{otherwise}
    			\end{cases}
    	\eqnnd
\enumd
\end{lemma}

\begin{proof}[Proof of Lemma~\ref{lemma.formulas}.]
(\ref{item.[m,L]}): $\{,\}$ is the graded commutator; since $m_\cC$ and $\theta L^{-(p+1)}$ have odd degree, this means $\{m^1_\cC, \theta L^{-(p+1)}\} = m^1_\cC \theta L^{-(p+1)} +\theta L^{-(p+1)}m^1_\cC$. 
We have
	\begin{align*}
		m^1_\cC (\theta L^{-(p+1)})(\eta + \theta \xi) = m^1_\cC( \theta L^{-(p+1)}\eta) = \omega L^{-(p+1)}\eta - \theta dL^{-(p+1)} \eta\nonumber
	\end{align*}
and
	\begin{align*}
		\theta L^{-(p+1)} m^1_\cC (\eta + \theta \xi)
	= \theta L^{-(p+1)} (d \eta + \omega \xi - \theta d\xi) = \theta(L^{-(p+1)} d \eta + L^{-(p+1)}\omega \xi).
	\end{align*}
The result follows.

(\ref{item.[d,L]}): Decomposing
	\eqnn
		\eta =  \alpha_k + \omega^{p+1} \beta_{k-2p-2} + \ldots
	\eqnnd
we have that
	\eqnn
	dL^{-(p+1)} \eta 
		= d(\beta_{k-2p-2} + \ldots) 
		= \del_+ \beta_{k-2p-2} + \omega (\del_- \beta_{k-2p-2} + \del_+ \beta_{k-2p-4}) + \ldots
	\eqnnd
while
	\eqnn
	L^{-(p+1)} d \eta 
		= L^{-(p+1)} ( \ldots + \omega^{p+1}(\del_- \beta_{k-2p} + \del_+ \beta_{k-2p-2}) + \ldots)
	\eqnnd
so we have
	\eqnn
	[d, L^{-(p+1)}] \eta = -\del_- \beta_{k-2p}
	\eqnnd
which is equal to $-L^{-(p+1)} d\,\pip \eta$ by definition of $\pip$. 

When $|\eta| >n+p$, let $|\eta| = n+p +r$, with $r > 0$.  Then,
	\begin{align*}
		(dL^{-(p+1)}- L^{-(p+1)}d) \eta
		&= (dL^{-(p+1)}- L^{-(p+1)}d) (\omega^{p+r}\beta_{n-p-r} + \omega^{p+r+1} \beta_{n-p-r-2} + \ldots) \\
		&= d(\omega^{r-1}\beta_{n-p-r} + \omega^{r}\beta_{n-p-r-2} + \ldots) 
		\\ & \qquad- L^{-(p+1)} (\omega^{p+r+1} (\del_- \beta_{n-p-r} + \del_+ \beta_{n-p-r-2}) + \ldots )\nonumber\\
		&= \omega^{r-1} \del_+ \beta_{n-p-r}\nonumber
	\end{align*}
which is what we claimed.

(\ref{item.g2ab}): By definition, we have that $g^2 = -\theta L^{-(p+1)} m^2_\cC (g^1 \tensor g^1)$. Note that if any term of $m^2_\cC(g^1 \tensor g^1)$ has a factor of $\theta$ in it, the operation $\theta L^{-(p+1)}$ renders it zero since $\theta \wedge \theta = 0$ in $\cC$. On the other hand, $g^1 a$ for $a \in \cF$ has a term without a $\theta$ only when $|a| < n+p$. 
\end{proof}

\begin{proof}[Proof of Theorem~\ref{thm.Aoo-map}]
We need to show that the sequence of maps $g^l : \cF^{\tensor l} \to \cC[1-l]$ satisfy the $A_\infty$ map condition  of \eqref{Ainmap}  	
	\eqnn
	\sum_{r + s + t = n} (-1)^{r + st}\,g^l (1^{\tensor r} \tensor m^s_\cF \tensor 1^{\tensor t})
	=
	\sum_{i_1 + \ldots + i_q = n} (-1)^{u}\,m^q_\cC(g^{i_1} \tensor \ldots g^{i_q}), \qquad n \geq 1
	\eqnnd
where $l=r+1+t$ and the sign on the right-hand side is given by
	\eqnn
	u = (q - 1)(i_1 - 1) + (q-2) (i_2 - 1) + \ldots + 2(i_{q-2}-1) + (i_{q-1}-1).
	\eqnnd
Noting that $m^k_\cC = 0$ for $k \geq 3$ and  $m^k_\cF = 0 $ for $k \geq 4$, and further, that $g^k = 0$ for $k \geq 3$, \eqref{Ainmap} becomes the following four equations:
	\begin{align}
			m^1_\cC\, g^1 
				= &g^1 m^1_\cF \tag{n=1}
			\\
			m^1_\cC \,g^2 
				+ g^2(m^1_\cF \tensor 1 + 1 \tensor m^1_\cF)+ m^2_\cC(g^1 \tensor g^1)
				= &g^1m^2_\cF \tag{n=2}\\
			g^2( 1 \tensor m^2_\cF -m^2_\cF \tensor 1)+ m^2_\cC(g^1 \tensor g^2 - g^2 \tensor g^1)=&g^1 m^3_\cF 
				 \tag{n=3} \\
			g^2(1 \tensor m^3_\cF + m^3_\cF \tensor 1) + m^2_\cC(g^2 \tensor g^2) = & 0~ .\tag{n=4}
	\end{align}
	
{\bf The $n=1$ equation.} We verified $g = g^1$ is a chain map in the proof of Lemma~\ref{lemma.chain}. 

{\bf The $n=2$ equation.} We first simplify the expression involving $g^2$ terms:
	\begin{align}
	m_\cC^1 &g^2 + g^2(m^1_\cF \tensor 1 + 1 \tensor m^1_\cF)\\
		&= m_\cC^1 (-\theta L^{-(p+1)} m^2_\cC(g^1 \tensor g^1)) + (-\theta L^{-(p+1)} m^2_\cC(g^1 \tensor g^1))(m^1_\cF \tensor 1 + 1 \tensor m^1_\cF) \nonumber \\
		&= -m_\cC^1 (\theta L^{-(p+1)}) m^2_\cC(g^1 \tensor g^1) + (-\theta L^{-(p+1)} m^2_\cC)(m^1_\cC \tensor 1 + 1 \tensor m^1_\cC)(g^1 \tensor g^1) \nonumber \\
		&= -m_\cC^1 (\theta L^{-(p+1)}) m^2_\cC(g^1 \tensor g^1) - (\theta L^{-(p+1)} m^1_\cC m^2_\cC)(g^1 \tensor g^1) \nonumber \\
		&= -\left(m_\cC^1 (\theta L^{-(p+1)}) + \theta L^{-(p+1)} m_\cC^1 \right) m^2_\cC(g^1 \tensor g^1)  \nonumber \\
		&= -\{m_\cC^1 ,\theta L^{-(p+1)}\}m^2_\cC(g^1 \tensor g^1)~.  \nonumber
	\end{align}
We have used that $g^1$ is a chain map, and that the operations $m^1_\cC$ and $m^2_\cC$ satisfy the Leibniz rule. So the left-hand side of the $n=2$ equation becomes
\begin{align}
m_\cC^1 g^2 + g^2(m^1_\cF \tensor 1 + 1 \tensor m^1_\cF) + m^2_\cC(g^1 \tensor g^1)
		&= \left(-\{m_\cC^1 ,\theta L^{-(p+1)}\} + 1\right) m^2_\cC(g^1 \tensor g^1)~.
\end{align}

Let $a_1, a_2$ be elements of $\cF$.  Let us write
	\eqnn
	m^2_\cC(g^1 \tensor g^1)(a_1 \tensor a_2) = \eta + \theta\, \xi .
	\eqnnd
Then, by Lemma~\ref{lemma.formulas}(\ref{item.[m,L]}), we find
	\begin{align}
	&\left(m_\cC^1 g^2 + g^2(m^1_\cF \tensor 1 + 1 \tensor m^1_\cF)+m^2_\cC(g^1 \tensor g^1) \right)(a_1 \tensor a_2) \nonumber\\
		&\qquad\qquad= \left(-\{m^1_\cC,\theta L^{-(p+1)}\}+1 \right) (\eta + \theta \xi) \nonumber \\
		&\qquad\qquad= (1- \omega^{p+1} L^{-(p+1)}) \eta + \theta \left([d,L^{-(p+1)}] \eta + (1- L^{-(p+1)} \omega^{p+1}) \xi \right) \nonumber\\
		&\qquad\qquad= \pip \eta + \theta \left([d,L^{-(p+1)}] \eta + \pips \xi \right) . 
		\label{eqn.n2left}
	\end{align}
The expression for $\eta$ and $\xi$ are found by calculating $m^2_\cC(g^1 \tensor g^1)(a_1 \tensor a_2)$ explicitly.   Recalling the definitions of $g^1$ in Definition~\ref{defn.fg}, we find that $m^2_\cC(g^1 \tensor g^1)(a_1 \tensor a_2)$ equals
\begin{align*}\label{eqn.m2g}
			\begin{cases}
			a_1 \w a_2 - \theta(L^{-(p+1)} da_1 \w a_2 + (-1)^{a_1} a_1 \w L^{-(p+1)}d a_2)
				& |a_1|,|a_2| \leq n+p \\
			-  \theta ( (-1)^{a_1} a_1 \w \ast_r a_2) & |a_1| \leq n+p, |a_2| > n+p \\ 
			- \theta (\ast_r a_1) \w a_2
				& |a_1| > n+p, |a_2| \leq n+p \\
			0 & |a_1|,|a_2| > n+p.
			\end{cases}
\end{align*}
We can now read off the expression for $\eta$ and $\xi$ and substitute them  into \eqref{eqn.n2left}.  In the case of $|a_1|,|a_2| \leq n$, we have 
\begin{align}
\eta = a_1 \w a_2 ~, \qquad \xi= - L^{-(p+1)} da_1 \w a_2 - (-1)^{a_1} a_1 \w L^{-(p+1)}d a_2~.
\end{align}
Using Lemma \ref{lemma.formulas}\eqref{item.[d,L]}, we find in this case that
\begin{align*}
\pip \eta &+ \theta \left([d,L^{-(p+1)}] \eta + \pips \xi \right)\\ 
& =
\begin{cases}
 \pip (a_1\w a_2) - L^{-(p+1)} d\,\pip(a_1\w a_2)   & |a_1| + |a_2| \leq n+p \\
 \theta \pips \big[dL^{-(p+1)}(a_1\w a_2)&  |a_1| + |a_2| > n+p \\ 
 \qquad\quad - ( L^{-(p+1)} da_1) \w a_2 - (-1)^{a_1} a_1 \w L^{-(p+1)}d a_2\big] &  
\end{cases}
\end{align*}
The other three cases in \eqref{eqn.n2left} are straightforward.  All in all, we find that the left-hand side of the $n=2$ condition matches exactly the right-hand side (using the definition of $m^2_{\cF}$ in Section~\ref{section.AooF}) given by
\begin{align*}\label{eqn.g1m2}
g^1&m^2_\cF(a_1 \tensor a_2) \\
		&=
		\begin{cases}
			\pip(a_1 \w a_2) - \theta L^{-(p+1)}d\, \pip(a_1 \w a_2) & |a_1| + |a_2| \leq n+p \\
			-\theta\, \pips (-dL^{-(p+1)} (a_1 \w a_2)& |a_1|+|a_2| > n+p\,\\
			 \qquad \qquad + L^{-(p+1)}da_1 \w a_2 + (-1)^{a_1} a_1 \w L^{-(p+1)}d a_2) 
			 &~  {with}~~ |a_1|,|a_2| \leq n+p \\
			-\theta\, \pips ( (-1)^{a_1} a_1 \w \ast_r a_2) & |a_1| \leq n+p, |a_2| > n+p \\
			-\theta\, \pips (\ast_r a_1 \w a_2) & |a_1| > n+p, |a_2| \leq n+p \\
			0 & |a_1|,|a_2| > n+p
		\end{cases}
\end{align*}

{\bf The $n=3$ equation.} Let  $a_1, a_2, a_3$ be elements of $\cF$.  Consider first the left-hand side of the equation:
\begin{align*}
[g^2( 1 \tensor m^2_\cF -m^2_\cF \tensor 1)+ m^2_\cC(g^1 \tensor g^2 - g^2 \tensor g^1)](a_1 \tensor a_2 \tensor a_3)~.
\end{align*}
Notice that each term vanishes unless $|a_1|, |a_2|, |a_3| \leq n+p$.  For the first term, this is due to Lemma \ref{lemma.formulas}\eqref{item.g2ab}.   For the second term, since the $g^2$ map contains a $\theta$, a $\theta$ in the $g^1$ map term would result in a null $m^2_\cC$.  This forces the element which $g^1$ acts on to have degree less than or equal to $n+p$.   Hence, we only need to consider the case where $|a_1|, |a_2|, |a_3| \leq n+p$.  We note that this is a necessary condition for $m^3_\cF(a_1 \tensor a_2 \tensor a_3)$ to be non-trivial.  

For the first term, we have 
\begin{align*}
g^2( &1 \tensor m^2_\cF -m^2_\cF \tensor 1)(a_1 \tensor a_2 \tensor a_3)\\ 
&= -\theta L^{-(p+1)} m^2_\cC\left[ a_1 \tensor (m^2_\cF(a_2\tensor a_3)) -  m^2_\cF(a_1\tensor a_2) \tensor a_3\right] \\
&= -\theta L^{-(p+1)}\left[a_1 \w \pip(a_2\w a_3) -\pip(a_1\w a_2) \w a_3\right]\\
&=-\theta   L^{-(p+1)}L^{p+1}\left[a_1 \w L^{-(p+1)}(a_2\w a_3) - L^{-(p+1)}(a_1\w a_2) \w a_3\right]
\end{align*}
where in the third line, we omitted the term including $\theta$ in the $m^2_\cF$, and in the third line, we used the identity $\pip = 1 - L^{p+1}L^{-(p+1)}$.

For the second term, we have
\begin{align*}
m^2_\cC(&g^1 \tensor g^2 - g^2 \tensor g^1)(a_1 \tensor a_2 \tensor a_3) \\
&=m^2_\cC \left[g^1 \tensor (-\theta L^{-(p+1)})m^2_\cC(g^1 \tensor g^1) - (-\theta L^{-(p+1)})m^2_\cC(g^1 \tensor g^1)\tensor g^1\right](a_1 \tensor a_2 \tensor a_3) \\
&=-\theta\left[a_1 \w L^{-(p+1)}(a_2\w a_3)  -  L^{-(p+1)}(a_1\w a_2) \w a_3\right] 
\end{align*}
Combining the two expressions, we find for the left-hand side of the equation
\begin{align*}
[g^2( &1 \tensor m^2_\cF -m^2_\cF \tensor 1)+ m^2_\cC(g^1 \tensor g^2 - g^2 \tensor g^1)](a_1 \tensor a_2 \tensor a_3)\\
&= -\theta (1 -  L^{-(p+1)}L^{p+1}) \left[a_1 \w L^{-(p+1)}(a_2\w a_3)  -  L^{-(p+1)}(a_1\w a_2) \w a_3\right]  \\
& = - \theta \, \pips \left[a_1 \w L^{-(p+1)}(a_2\w a_3)  -  L^{-(p+1)}(a_1\w a_2) \w a_3\right] \\
& = g^1 m^3_\cF (a_1 \tensor a_2 \tensor a_3)
\end{align*}

{\bf The $n=4$ equation.} This equation is satisfied because every term is identically zero. To wit:
	\begin{itemize}
		\item
			$g^2(1 \tensor m^3_\cF + m^3_\cF \tensor 1)$ is zero for the following reason: $m^3_\cF$ is non-zero only when $m^3_\cF$ outputs an element of degree $\geq n +p +1$. Thus, by Lemma~\ref{lemma.formulas}(\ref{item.g2ab}), $g^2(1 \tensor m^3_\cF)$ and $g^2(m^3_\cF \tensor 1)$ are both identically zero.
		\item By definition, the output of $g^2$ involves a $\theta$, an odd degree form.   Thus, $m^2_\cC(g^2 \tensor g^2)$ vanishes since it involves the wedge product of $\theta$ with itself.
	\end{itemize}
\end{proof}

\section{Calabi-Yau structure on \texorpdfstring{$\cF$}{F}}
What we call a Calabi-Yau structure follows the terminology of~\cite{costello}, though this notion has also been called a {\em proper} Calabi-Yau structure~\cite{gps}, a {\em right} Calabi-Yau structure~\cite{brav-dyckerhoff}, and a ``symplectic structure on a formal pointed dg-manifold''~\cite{kontsevich-soibelman-Aoo}.\footnote{This is to contrast it with a different notion of Calabi-Yau-ness, which has been called a {\em smooth} Calabi-Yau structure~\cite{gps} and a left Calabi-Yau structure~\cite{brav-dyckerhoff} elsewhere. A category or algebra may be smoothly Calabi-Yau but not properly Calabi-Yau and vice versa; the distinction is important, for instance, as these different structures gives rise to two-dimensional field theories that are only partially defined, and the class of cobordisms on which the field theories are defined differs by the type of Calabi-Yau-ness one equips onto the category. We refer the reader to Section~6 of~\cite{gps} for a more thorough account.}

\begin{defn}\label{CY-swap}
An $A_\infty$-algebra $A$ over a base ring $k$, together with a bilinear map
	\eqnn
	\langle \, , \, \rangle
	:	A \tensor A \to k[-D]
	\eqnnd
is called {\em Calabi-Yau of dimension $D$} if
	\begin{itemize}
		\item
			The pairing is non-degenerate on cohomology, 
		\item
			$\langle a, b \rangle = (-1)^{|a| |b|} \langle b, a \rangle$, and
		\item $\langle \, , \, \rangle$ satisfies the cyclic symmetry equation
			\eqn\label{eqn.cy}
			\langle m^{l} (a_1,\ldots,a_{l}), a_{l+1} \rangle
			=
			(-1)^{l + |a_1| (|a_2|+ \ldots + |a_{l+1}|)}
			\langle m^{l} (a_2,\ldots,a_{l+1}), a_1\rangle .
			\eqnd
	\end{itemize}
\end{defn}

\begin{example}\label{example.oriented-compact}
Let $X$ be a compact oriented manifold of dimension $D$. Then the cdga $\Omega(X)$ can be given the structure of a Calabi-Yau algebra of dimension $D$ by the following pairing:
	\eqnn
		\eta \tensor \xi \mapsto \int_X \eta \wedge \xi.
	\eqnnd
Note that if $\eta$ and $\xi$ do not have complementary degree, the integral vanishes.
\end{example}

\begin{remark}
Calabi-Yau algebras were first defined by Kontsevich to formalize properties of $\dbcoh$ of a smooth Calabi-Yau variety. It is a theorem of Costello~\cite{costello} that Calabi-Yau algebras, and Calabi-Yau categories more generally, define a topological CFT. For further discussion, see also Ginzburg~\cite{ginzburg}, and the work of Brav-Dyckerhoff~\cite{brav-dyckerhoff} generalizing the notion to {\em relative} Calabi-Yau structures.
\end{remark}

If our symplectic manifold $M$ is compact and $[\omega]$ is an integral class in cohomology, then $\cC$ is also a Calabi-Yau algebra. This is because the mapping cone $\cone(\omega^{p+1})$ can be identified (as a cdga) with the minimal model for the the cdga $\Omega(E_p)$---the differential forms on the total space $E_p$ of the sphere bundle associated to $\omega^{p+1}$.  (See Theorem~\ref{theorem.appendix}.)  Furthermore, since $\omega$ is symplectic, $E_p$ comes with an orientation. If $M$ is compact, so is $E$, hence Example~\ref{example.oriented-compact} implies that $\Omega(E_p) \simeq \cC$ is Calabi-Yau. 

The algebra equivalence of $\cC$ and $\cF$ however holds for all $\om$ regardless of its integrality.  Since the Calabi-Yau property is algebraic in nature, it is reasonable to expect that whether $\cF$ exhibits the Calabi-Yau property or not should not depend on the integrality of $[\om]$.  Indeed, we can prove directly on $M$ that $\cF$ is a Calabi-Yau algebra for all compact $(M^{2n}, \om)$.   

To define the pairing on $\cF$, we make use of the map $\cpF \to \RR[-2(n+p)-1)]$ given by
	\eqnn
	\BFO{0} \xra{\rstar} \Omega^{2n}(M) \xra{\int_M} \RR.
	\eqnnd
For degree reasons, this map is zero on all other degree components of $\cF$.   The map gives us a natural pairing for $\cF$.  

\begin{defn}[The Calabi-Yau structure on $\cF$]\label{defn.F-CY}
Define the pairing on $\cF$ by the composition
\eqnn
\xymatrix{
\cF \tensor \cF \ar[dr]_-{\langle \, , \, \rangle\quad} \ar[r]^-{m^2} &\cF \ar[d]^-{\int_M \rstar}\\
 & \RR[-2(n+p)-1)].
}
\eqnnd
Explicitly, for $a_1\in \FO{k}$ and $a_2 \in \BFO{k}$,
	\eqn\label{eqn.pairing}
	\langle a_1, a_2 \rangle
	= \int_M \rstar [m^2(a_1, a_2)] = 
	\int_M (-1)^k a_1 \wedge \rstar\, a_2 = \int_M  \rstar\, a_2 \, \wedge a_1  = \langle a_2, a_1 \rangle~.
	\eqnd
\end{defn}

\begin{remark}\label{remark.symmetry}
Note that $m^2$ is already graded-commutative, so the second condition of Definition \ref{CY-swap} is automatically satisfied. It is also worth noting that if $a_1 \in \FO{k_1}$ and $a_2 \in\BFO{k_2}$, we must have that $k_2 = 2n + 2p +  1 - k_1$ if $\langle a_1, a_2 \rangle$ is non-zero.  In this case, $|a_1 | | a_2| = 0\mod 2$, and therefore, $\langle a_1, a_2\rangle  = \langle a_2, a_1 \rangle$ as consistent with \eqref{eqn.pairing}.
\end{remark}

The pairing \eqref{eqn.pairing} defined on $\cF$ is in fact compatible to the Calabi-Yau pairing $\langle\, , \, \rangle_{\cC}$ on $\cC$ induced from $\Om(E)$ when $\om$ is integral.  In this case, we have the following commutative diagram:
	\eqnn
		\xymatrix{
			\cF \tensor \cF \ar[drr]^-{-\langle \, , \, \rangle} \ar[d]_-{g\, \tensor\, g} & & \\
			\cC \tensor \, \cC \ar[rr]_-{\langle \, , \, \rangle_{\cC}}&& \RR[-2(n+p)-1]
		}
	\eqnnd
To verify this, consider $\alpha_k \in \cF^k$ and $\alpha'_k \in \cF^{2n+2p+1-k}$ of complementary degrees.  Recalling the definition of the chain map $g$ in Definition \ref{defn.fg}, we have
\begin{align*}
\langle g\alpha_k , g \alpha'_k \rangle_{\cC} &=
		\int_{E_p} (\alpha_k - \theta_{2p+1} L^{-(p+1)} d \alpha_k) \wedge (- \theta_{2p+1} \ast_r \alpha'_k)\\
		&=
		-\int_{E_p} \theta_{2p+1} \wedge (-1)^k \alpha_k \wedge \ast_r \alpha'_k\,\\
		&= - \int_{M}  (-1)^k \alpha_k \wedge \ast_r \alpha'_k = - \langle \alpha_k, \alpha'_k \rangle
\end{align*}
where in the last second line, we have noted that the equivalence of $\cone(\omega^{p+1}) \simeq \Omega(E_p)$ is compatible with the Thom isomorphism for differential forms; namely, $\theta$ represents the Thom class, and for any differential form $A$ on $M$,
\begin{align}\label{EM-integral-relation}
		\int_{E_p} \theta_{2p+1} \wedge A = \int_M A.
\end{align}

We now show that pairing on $\cF$ satisfy the Calabi-Yau algebra conditions for any symplectic structure. 
\begin{theorem}
Definition~\ref{defn.F-CY} endows $\cpF^\bullet$ with a Calabi-Yau structure of dimension 2(n+p)+1.
\end{theorem}

\def\ddps{{d_+^{\ast}}}
\def\ddms{{d_-^{\ast}}}

\begin{proof}
{\bf Non-degeneracy.} The non-degeneracy of the pairing \eqref{eqn.pairing} on cohomology was previously shown in two special cases: (i) for $p=0$ and $k\neq n$ in Proposition~3.7 of~\cite{tseng-yau-2}; (ii) for $k=n+p$  in Proposition~3.26 of~\cite{tseng-yau}.  We will give the argument here for the remaining cases of $k<n+p$ (which follow the same techniques as the previous cases). 

We make use of a compatible triple, $(\om, J, g)$, introducing a compatible almost complex structure, $J$, and the associated Riemannian metric, $g$, on $(M^{2n}, \om)$.  Since the filtered cohomologies $F^pH(M)$ are associated with elliptic complexes (see Theorem \ref{ellipticC}), there is an elliptic Laplacian associated with each cohomology.  By Hodge theory, there then exists a unique harmonic representative for each cohomology class.  The harmonic representative for $k<n+p$ satisfies the following:
\begin{alignat*}{3}
\ddp \,a_1 &=0\,, \qquad \qquad \quad\ddm \,a_2 &&=0\,, \\
\ddps\,  a_1 &=0\,, \qquad \qquad\quad \ddms\, a_2 &&=0\,.
\end{alignat*}
for $a_1 \in \FO{k}$ and $a_2 \in \BFO{k}$ 
and the adjoint here is defined with respect to the standard inner product on differential forms 
\begin{align*}
(\eta_k , \,\xi_k) = \int_M \eta_k \w \ast\, \xi_k 
\end{align*}
where $\eta_k, \xi_k \in \Omega^k(M)$ and, $\ast\,$, the Hodge star operator.  Let us recall that the action of the Hodge star on the Lefschetz decomposed element, $\om^r \beta_s$, where $\beta_s\in P^s(M)$, is given by (see for example \cite{tseng-yau})
\begin{align*}
\ast \; \om^r \beta_s  = (-1)^{s(s+1)/2} \dfrac{r!}{(n-r-s)!}\, \om^{n-r-s} \cJ\beta_s 
\end{align*}
where
\begin{align*}
\cJ=\sum_{p,q} (\sqrt{-1}\,)^{p-q} \ \Pi^{p,q}
\end{align*}
with $\Pi^{p,q}$ defined to be the projection of a form into its $(p,q)$ components with respect to the almost complex structure $J\,$.  Introducing further the operator $\cR$ defined to be
\begin{align*}
\cR (\om^r \beta_s) := (-1)^{s(s+1)/2} \dfrac{r!}{(n-r-s)!} \,\om^r \beta_s\,,
\end{align*}
we then have the relation
\begin{align}\label{starrel}
 \ast =  \rstar \cR \cJ\,.
\end{align}
We can now express $\{\ddp, \ddm\}$ and their adjoints  acting on the space of $p$-filtered forms, $F^p\Om(M)$, as follows: 
\begin{alignat*}{3} 
\ddp &= \Pi^p\, d\,,  \qquad \qquad\quad\; \ddm &&= \rstar \, d \; \rstar\,,\\
\ddps & = - \ast \, d \;\ast \,,\qquad\qquad \ddms && = - \cR \cJ \, d\, \cR \cJ \, .
\end{alignat*}
Now, if $a_1 \in \FO{k}(M)$ is a harmonic representative of $F^pH^k_+(M)$, then $\ddm (\cR \cJ a_1) = \rstar\, d\, \ast a_1 =0\,,$ that is, $\cR \cJ (a_1)$ represents a cohomology class of $F^pH^k_-(M)$.  Similarly, if $a_2\in \BFO{k}(M)$ is a harmonic representative of $F^pH^k_-(M)$, then $\cR \cJ(a_2)$ represents a cohomology class of $F^pH^k_+(M)$.  Finally, to prove non-degeneracy of the pairing \eqref{eqn.pairing} on cohomology, choose $a_1$ to be a non-trivial harmonic representative of $F^pH^k_+(M)$.  Then setting $a_2 = \cR \cJ a_1$, we have
\begin{align*}
 \langle a_1, a_2 \rangle = \int_M  (-1)^k a_1 \w \rstar\, a_2 =(-1)^k \int_M  a_1 \w \rstar \cR\cJ a_1 = (-1)^k \|a_1\|^2 \neq 0\,,
 \end{align*} 
having applied \eqref{starrel}.  Similarly, for $a_2$ a non-trivial harmonic representative of $F^pH^k_-(M)$, we take $a_1 = \cR\cJ a_2$ and this results in
\begin{align*}
 \langle a_2, a_1 \rangle = \int_M \rstar\, a_2 \w a_1 = \int_M \rstar\, a_2  \w \cR\cJ a_2 = \int_M a_2  \w \rstar \cR\cJ a_2 = \|a_2\|^2 \neq 0\,,
\end{align*}
thereby proving the non-degeneracy of the pairing on $F^pH^k_{\pm}(M)$ for $k<n+p$.

{\bf Symmetry.} This is Remark~\ref{remark.symmetry}.

{\bf Cyclic symmetry for $l=1$.}  There are two cases that we need to consider.  First, fix $a_1 \in \FO{k-1}$,  $a_2 \in \BFO{k}$ and let $k\leq n+p$. Then the cyclic condition (\ref{eqn.cy}) becomes:
	\begin{align}
	\langle m^1 a_1,  a_2 \rangle
	&= \int_M \rstar\, a_2 \wedge \ddp a_1 \nonumber \\
	&= \int_M \rstar \,a_2 \wedge d a_1   \label{eqn.cy-check-1}\\
	&= (-1)^{k+1} \int_M d \rstar a_2 \wedge a_1 \label{eqn.cy-check-2}\\
	&= (-1)^k \int_M \rstar (-\ddm) a_2 \w a_1  = (-1)^{1+|a_1||a_2|} \langle m^1 a_2, a_1 \rangle	\nonumber
	\end{align}	
Here, (\ref{eqn.cy-check-1}) holds because $\ddp a_1 =  \pip d a_1 = da_1 - L^{p+1} L^{-(p+1)} da_1$, and since $a_2$ is a $p$-filtered form, $L^{p+1} \rstar a_2 = 0$. \eqref{eqn.cy-check-2} follows from integration by parts.  Note also that $|a_2| = 2(n+p)+1-k$ and therefore $|a_1||a_2| = k\mod 2$.  

For the second case, let $a_1, a_2 \in \FO{n+p}$.  Then,
	\begin{align*}
	\langle m^1 a_1, a_2 \rangle
	&=   \int_M \rstar (-d\Lambda d) a_1 \w a_2 = \int_M (-d\Lambda d)  a_1 \w \rstar a_2 \\
	&= (-1)^{n+p+1} \int_M \Lambda da_1 \w d \rstar a_2 =  (-1)^{n+p+1} \int_M  da_1 \w \Lambda d \rstar a_2 \\
	& = \int_M a_1 \w d\Lambda d\rstar a_2  = (-1)^{1+(n+p)^2} \int_M \rstar (-d\Lambda d) a_2 \w a_1 \\
	&= (-1)^{1+|a_1||a_2|} \langle m^1 a_2, a_1 \rangle
	\end{align*}
where we have noted that $\dpp\dpm = d\Lambda d$ which commutes with $\rstar$,  and also, for any $\eta\in \Omega^{k+2}$ and $\eta'\in \Omega^{2n-k}$,  
	\begin{align}\label{lambda-integ}
	\int_M \Lambda \eta \w \eta' = \int_M \eta \w \Lambda \eta'
	\end{align}
 
{\bf Cyclic symmetry for $l=2$.} The relation is
	\begin{align*}
	\langle m^2(a_1, a_2), a_3 \rangle &= (-1)^{|a_1|(|a_2|+|a_3|)} \langle m^2(a_2, a_3), a_1\rangle \\
	&= \langle a_1, m^2(a_2, a_3)\rangle
	\end{align*}
Hence, we need to show 
	\begin{align}\label{cy21}
	\int_M \rstar \left[ a_1 \times (a_2 \times a_3) - (a_1 \times a_2) \times a_3 \right] = 0 
	 \end{align}
for $|a_1|+|a_2|+|a_3| = 2n+2p+1$.  This follows from integrating the $A_\infty$ relations for $m^3$. Explicitly, the $A_\infty$ relation relevant here is
	\eqnn
m^2 (1 \tensor m^2)  - m^2(m^2 \tensor 1) = m^1(m^3) + m^3( m^1 \tensor 1 \tensor 1 + 1 \tensor m^1 \tensor 1 + 1 \tensor 1 \tensor m^1).
	\eqnnd
Plugging in elements $a_1, a_2, a_3$, the terms involving $m^2$ on the left-hand side are precisely those in (\ref{cy21}). So we can show instead that
 	\eqnn
	\int_M \rstar \left\{[m^1(m^3) + m^3( m^1 \tensor 1 \tensor 1 + 1 \tensor m^1 \tensor 1 + 1 \tensor 1 \tensor m^1) ](a_1 \tensor a_2 \tensor a_3)\right\} =0.
	\eqnnd
Notice that the first term vanishes since $m^3(a_1,a_2,a_3) \in \BFO{1}$ and therefore 
	\eqnn
	\int_M \rstar \left[m^1 m^3(a_1,a_2,a_3)\right] = -\int_M \rstar\, \ddm [m^3(a_1,a_2,a_3)] = -\int_M  d \left[\rstar  m^3(a_1,a_2,a_3)\right] = 0
	\eqnnd
by Stokes's Theorem. For the remaining terms, the only possible non-trivial  contribution of the integrand takes the form
	\begin{align}
	 & \ddp a_1 \wedge L^{-(p+1)}(a_2 \wedge a_3) - L^{-(p+1)}(\ddp a_1 \wedge a_2) \wedge a_3 \nonumber \\
	 & + (-1)^{i}\left[a_1 \wedge L^{-(p+1)}(\ddp a_2 \wedge a_3) - L^{-(p+1)}(a_1 \wedge \ddp a_2) \wedge a_3\right] \label{cy22}\\
	 & +(-1)^{i+j}\left[a_1 \wedge L^{-(p+1)}(a_2 \wedge \ddp a_3) - L^{-(p+1)}(a_1 \wedge a_2) \wedge \ddp a_3\right] \nonumber
	\end{align}
where $|a_1|, |a_2|, |a_3| <n+p$.  In fact, each term above is zero.  For instance, consider the first term in \eqref{cy22}.  To take account of the $L^{-(p+1)}$ operator, we can express
	\begin{align*}
	\alpha_j \w \alpha_k  =  \om^{n+2p+1-i} \w \eta_{i-2p-1} 
	\end{align*}
where  $\eta_{i-2p-1}\in \Om^{i-2p-1}$ and $\alpha_l \in \FO{l}$. In above, we have used the fact that $i+j+k=2n+1$ and $i<n+p$.  Then, we have 
	\begin{align*}
	\ddp \alpha_i \w L^{-(p+1)}(\alpha_j \w \alpha_k) =  \om^{n+p-i} \w \ddp \alpha_i \w \eta_{i-2p-1} =0
	\end{align*}
by the $p$-filter condition $L^{n+p-i}( \FO{i+1}) =0$.  Similar computations for the other terms show that all the terms in \eqref{cy22} vanish identically.

{\bf Cyclic symmetry for $l=3$.}  This follows directly from the definition of $m^3$ and applying \eqref{lambda-integ}.  Since $m^l = 0$ for $l \geq 4$, this finishes the proof.
\end{proof}

\begin{remark}
To any $A_\infty$-algebra with a Calabi-Yau structure, one can associate a two-dimensional topological conformal field theory \cite{costello}.   The action of this field theory consists of a standard kinetic term and a potential term $\Phi$.  The potential term is determined by the $A_\infty$ structure and takes the form
\begin{align*}
\Phi(x) = \sum^\infty_{l=1} \dfrac{1}{l+1} \langle m^l(x, x, \ldots, x), x \rangle \,,
\end{align*}
where $x$ is a formal sum of elements of the $A_\infty$-algebra.  For instance, for the de Rham cdga in dimension $D=3$, with $x= \alpha_1$, $\Phi(x=\alpha_1)$ is the Chern-Simon functional   
$$\Phi(x) = \dfrac{1}{2}  \langle d\alpha_1,  \alpha_1\rangle +  \dfrac{1}{3}\langle \alpha_1 \w \alpha_1, \alpha_1 \rangle\,. $$
It is also straightforward to write down the potential for $\cF$.  For example, in dimension four and with $x = \beta_1 + \beta_2$, where $\beta_i\in F^0\Omega^i(M)$, one finds
$$\Phi(x) = \langle m^2(\beta_2, \beta_2), \beta_1 \rangle + \dfrac{1}{2} \langle m^3(\beta_1, \beta_1, \beta_2), \beta_2 \rangle\,.$$ 
\end{remark}

\section{Homology and intersection when \texorpdfstring{$\omega$}{omega} is integral}
\label{section.intersection}

In this section, we will assume $\omega \in H^2(M, \mathbb{Z})$ so that $E_p$ is a smooth manifold.  On $E_p$, there is the standard homology and intersection theory on it.  The equivalence of  $\Om(E_p) \simeq \cF_p(M)$ suggests that there should also be a dual homology and intersection theory on $(M, \om)$, and one that is distinctively symplectic in nature.  Here, we will provide a description of what such a theory should look like on $M$.  For simplicity, we will focus on the $p=0$ case, and in this case, the objects involved turn out to be coisotropic and isotropic spaces on $M$.  Due to its close association with forms, we will use the language of currents to describe these spaces.  Admittedly, our description here is heuristic in nature and will require further investigations which we plan to pursue in the future.   

Since $p$ will be taken to be zero for the most part, we will drop the subscript label in $E_p$ and $\cpF$ for the rest of this section.

\subsection{Recollections on currents}
Recall that the collection $\cD_{m}$ of $m$-dimensional currents is defined to be the weak $\RR$-linear dual to the space of compactly supported smooth $m$-forms on a manifold of dimension $d$. Given a current $\rho$, its boundary $\del\rho$ is defined by dualizing the de Rham differential:
	\eqnn
		\del \rho(\eta) := \rho (d\eta).
	\eqnnd
A form of the de Rham theorem states that the resulting chain complex of currents
	\eqnn
		0 \to \cD_{d} \xra{\del} \cD_{d-1} \xra{\del} \ldots \xra{\del} \cD_0 \to 0 \ldots 
	\eqnnd
has homology isomorphic to singular homology, provided that the manifold $M$ is connected, compact, and oriented:
	\eqnn
		H_*(\cD) \cong H_*(M;\RR).
	\eqnnd
\begin{remark}
Since $\cD_m$ is a very large vector space, one can often restrict to smaller subspaces while still retaining the same homology. For example, one can show that currents that are represented by Lipschitz neighborhood retracts, and whose $\del$ are also represented by such, have the same cohomology. (This follows by observing that both theories satisfy the Eilenberg-Steenrod axioms; see for instance Federer~\cite{federer}.)
\end{remark} 

\begin{example} 
Note that oriented, compact $m$-manifolds $X \subset M$ (possibly with boundary) define distributions by sending $\eta \mapsto \int_X \eta$. We also note that, when $M$ is compact and oriented of dimension $d$, any $(d-m)$-form $\rho$ defines a linear map by $\eta \mapsto \int_M \eta \wedge \rho$. 

Indeed, given any $X$, one can define its dual current to be represented by a distributional form $\rho_X$ such that
	\eqnn
		\int_X \eta = \int_M \eta \wedge \rho_X\,.
	\eqnnd
Here, $\rho_X$ is a differential form with coefficients in distributions; the distribution coefficient is the Dirac delta function supported on $X$. 
\end{example}

\subsection{Homology of primitive currents}
Coisotropic and isotropic subspaces are related to primitive currents  by the following:
\begin{lemma}[Section 4 of \cite{tseng-yau}] \label{dualc}
Let $N\subset M$ be an embedded, compact, oriented submanifold with dual current $\rho_N$.  Then
\begin{itemize}
\item[(i)] $\rho_N$ is primitive if and only if $N$ is coisotropic;
\item[(ii)] $\rstar \rho_N$ is primitive if and only if $N$ is isotropic. 
\end{itemize}
\end{lemma}

To see how coisotropic and isotropic subspaces both arise in the $p=0$ case, we take the cochain complex of \eqref{eqn.A}  
\begin{align*}
	{\cF} 
	 = \left(P^0  \xra{\dpp} P^1  \xra{\dpp} \ldots
	\xra{\ddp} 	P^{n} \xra{-\dpp \dpm}
	 {\bar P}^{n}\xra{-\dpm} {\bar P}^{n-1}
	\xra{-\dpm}\ldots\xra{-\dpm}{\bar P}^0\right)
\end{align*}
and make a modification by replacing the ``bar" elements ${\bar F^0\Omega^k}= {\bar P^k}$ by $\rstar P^k = \om^{n-k}P^k$.  This results in the following complex:
\begin{align}\label{cochain0}
	{\tilde \cF} 
	=\left(P^0  \xra{\dpp} P^1  \xra{\dpp} \ldots
	\xra{\dpp} 	P^{n} \xra{-\dpp \dpm}
	 \om^0{P}^{n}\xra{-d} \om^1{P}^{n-1}
	\xra{-d}\ldots\xra{-d}\om^n{\bar P}^0\right).
\end{align}
Note that $\rstar\, \dpm \, \rstar = d\,$, which is just the standard exterior derivative when acting on $\om^{n-k} P^k$ forms.  

Following the intuition of~Lemma \ref{dualc}, the first half of ${\tilde \cF}$ can be roughly thought of as dual to a complex of coisotropic chains $C_\bullet$; the second half is dual to a complex of isotropic chains $I_\bullet$ since $\rstar (\om^{n-k} P^k) = P^k$.  In all, the cochain complex \eqref{cochain0} suggests that we seek a putative chain complex of coisotropic and isotropic currents:
\begin{align}\label{chain0}
	C_{2n}  \xra{\paa} C_{2n-1}  \xra{\paa} \ldots
	\xra{\paa} 	C_{n} \xra{\pab}
	 I_n\xra{\pa} I_{n-1}
	\xra{\pa}\ldots\xra{\pa}I_0.
\end{align}
Here, $C_k$ is the vector space generated by currents supported by coisotropic chains of dimension $k$, while $I_k$ is generated by currents supported on isotropic chains of dimension $k$, with $C_n=I_n=\cL$ being the space of Lagrangian currents. The boundary maps $\{\pa, \pab, \paa\}$ then have the following interpretations:
\begin{enumerate}
\item $\pa$ is dual to the exterior derivative, $d$, and therefore $\pa$ is just the standard boundary map.  Note that the boundary of an isotropic subspace is also isotropic.\\
\item $\pab$ is dual to $\dpp\dpm= d \circ \dpm$.  This is a ``boundary" operator mapping Lagrangians to Lagrangians.  One way to understand $\pab$ is to study first the dual of the operator $\dpm$.   Strikingly, the differential operator $\dpm$ {\it lowers} the degree of a current by one and so its dual would need to {\it increase} the dimension of a chain by one.  Such a novel operation can be motivated by noting that for a degree $n$ primitive current, $\beta_n$,  $d \beta_n = \om \w (\dpm \beta_n)\,$, which dualizes to  
$$\pa L = H \cap L_+\,,$$   
with $L$ being the Lagrangian chain dual to $\beta_n$, $H$ the $(2n-2)$-dimensional symplectic submanifold dual to $\om\,$, and $L_+$ being dual to $\dpm \beta_n\,$.\footnote{The existence of a symplectic submanifold dual to $[\om]$ an integral class is currently known only for $[k\om]$ where $k$ is a sufficiently large integer \cite{donaldson}.}
In words, $L_+$ is an $(n+1)$-dimensional chain that intersects with $H$ along the boundary of $L$.   Altogether, $\pab$ should correspond to the composition of $L \to L_+ \to \pa L_+$, i.e. it maps $L$ to the boundary of $L_+$.\\
\item $\paa$ is dual to $\dpp = \Pi^{0} \circ d$ where $\Pi^0:\Omega^k \to P^k$ is the projection operator onto the primitive component (see \eqref{projp}).  This suggests that $\paa = \pi^0 \circ \pa$ where $\pi^0$ heuristically ``projects" onto the coisotropic component of the boundary.  For a better intuition, we note that for $\beta_k$, a primitive $k$-current, $d\beta_k = \dpp \beta_k + \om \w (\dpm \beta_k)\,$.  This dualizes to 
$$\paa C = \pa C  -  H \cap C_+$$ 
where $C$ is the $(2n-k)$-dimensional coisotropic chain dual to $\beta_k$, and $C_+$ is the $(2n-k+1)$-dimensional coisotropic chain dual to $\dpm \beta_k$.  Hence, the projection $\pi^0$ effectively removes point $x\in \pa C$ whose tangent space $T_x(\pa C)\subset T_xH$.      
\end{enumerate}

The chain complex \eqref{chain0} implies several distinct homologies for isotropic and coisotropic chains on $(M^{2n}, \om)$.  For isotropic chains, we have 
\begin{align*}
H_k^{I}(M) := \dfrac{\ker(\pa\!: I_k \to I_{k-1})}{\im (\pa\!: I_{k+1} \to I_{k})}\, , \qquad 0\leq k \leq n-1\,,
\end{align*}
associated to the right hand side of \eqref{chain0}.  This is just standard homology but with objects restricted to isotropic chains.  For coisotropics chains on left hand side of \eqref{chain0}, we have the homology:
\begin{align*}
H_{2n-k}^{C}(M) := \dfrac{\ker(\paa\!: C_{2n-k} \to C_{2n-k-1})}{\im (\paa\!: C_{2n-k+1} \to C_{2n-k})}\, , \qquad 0\leq k \leq n-1\,.
\end{align*}
Of note, coisotropic elements in the homology can have boundary, but the boundary may not have any coisotropic component.  Lastly, at the center of chain complex \eqref{chain0}, we have two homologies involving Lagrangians:
\begin{align*}
H_n^{I}(M) &:= \dfrac{\ker(\pa: \cL_n \to I_{n-1})}{\im (\pab: \cL_{n} \to \cL_{n})}\, , \\
H_n^{C}(M) &=\dfrac{\ker(\pab: \cL_n \to \cL_{n})}{\im (\paa: C_{n+1} \to \cL_{n})}\, .
\end{align*}
These two Lagrangian homologies certainly can have elements that are not in the standard homology $H_n(M)$.  For $H_n^{I}(M)$, though the Lagrangian  generators must be boundaryless, a Lagrangian that is the boundary of an $(n+1)$-dimensional chain can be a non-trivial element.  Such a Lagrangian would only be trivial in $H_n^{I}(M)$ if it is also in the image of the $\pab$ map.  On the other hand, for $H_n^{C}(M)$, its elements may include Lagrangians which have non-trivial boundary but are trivial under the $\pab$ map.  Examples of both cases will be seen in the Kodaira-Thurston four-manifold discussed at the end of this section.

\subsection{Intersection theory via lifts to \texorpdfstring{$E$}{E}}
To better understand these homologies, we can make use of the isomorphism between the cohomology of primitive currents, $F^0H(M):=PH(M)$, and the singular cohomology of the contact circle bundle,  $H(E_0)$.  Then we can explicitly pull back the currents in $\tilde \cF(M)$ to currents in $\Om(E)$ by reinterpreting the map $g: \cF(M) \to \Om(E)$ from Definition~\ref{defn.fg} as a map $\tilde g: \tilde \cF(M) \to \Om(E)$. See Figure~\ref{figure.mapgz}.

\begin{figure}[t]
\begin{align*}
\xymatrix
{
 \ldots \ar[r]^-d & \cC^k \ar[r]^-d& \ldots\ar[r]^-d& \cC^{n}  \ar[r]^-d& \cC^{n+1} \ar[r]^-d& \ldots \ar[r]^-{d} & \cC^{2n+2p+1-k} \ar[r]^-{d} & \ldots \\
 \ldots \ar[r]^-{\dpp} & P^{k} \ar[u]^-{\beta_k-\theta \dpm\beta_{k}}\ar[r]^-{\dpp}& \ldots \ar[r]^-{\dpp} & P^{n} \ar[u]^-{\beta_{n} - \theta \dpm\beta_{n}}\ar[r]^-{-\dpp\dpm} & \om^{0}P^{n} \ar[u]^-{-\theta \beta_{n}}\ar[r]^-{-d} &\ldots  \ar[r]^-{-d} & {\om^{n-k}P}^{k} \ar[u]^-{-\theta\, \om^{n-k} \beta_k} \ar[r]^-{-d}& \ldots
}
\end{align*}
\caption{The map ${\tilde g}: {\tilde \cF}(M) \to \Om(E)$.}
\label{figure.mapgz}
\end{figure}

The interpretation of the map $\tilde g$ at the chain level is as follows:
\begin{itemize}
\item[$\bullet$] A coisotropic chain $C_{2n-k}$ of codimension $k$ in $M$ (represented by the dual current $\beta_k\in P^k(M)$), is mapped to a codimension $k$ chain $\tilde C_{2n+1-k}$ of dimension $(2n+1-k)$.  The extra dimension comes from $\tilde C$ wrapping around the circle fiber of $E$ plus another the addition of another component that is the dual of $\theta\dpm \beta_k$.  This additional component is important as those coisotropic chains that have non-trivial boundary, but yet the boundary does not have a coisotropic component, become boundaryless cycles after lifted to $E$ (since $\beta_k - \theta \dpm \beta_k$ is $d$-closed).
\item[$\bullet$] An isotropic chain $I_k$ of dimension $k$ in $M$ (represented by the dual current $\om^{n-k} \beta_k \in \om^{n-k} P^k$) gets mapped to a section $\tilde I_k$ of $E$ which is still of dimension $k$.  This is possible since $\omega|_I = 0\,$ and thus the restricted bundle $E|_I$ is trivial.  In particular, if $k=n$, then $I$ is a Lagrangian and the resulting $\tilde I$ in  $E$ is Legendrian.  Furthermore, if $d (\om^{n-k} \beta_k)=0$, this implies that $I$ has no boundary and neither does $\tilde I$ after the lift.
\end{itemize}
Under the map $\tilde g$, the homologies $\{H^C_{\bullet}(M), H^I_{\bullet}(M)\}$ map to the standard homology $\{H_{\bullet}(E)\}$.  Thus the Calabi-Yau pairing on $\cF$ is interpreted as the usual intersection pairing of $\tilde C_{2n+1-k}$ with $\tilde I_k$ inside $E$. 

\begin{subsection}{Example: Kodaira-Thurston four-fold.}\label{section.KT}
The Kodaira-Thurston manifold, $KT^4$, is a closed, non-K\"ahler, symplectic four-manifold.  It can be defined as the quotient of $\RR^4$ with coordinates $\{\xa, \xb, \xc, \xd\}$ under the identification
\begin{align}\label{qid}
(\xa, \xb, \xc, \xd) \sim (\xa + a, \xb + b , \xc + c, \xd + d - b\,\xc) 
\end{align}
where $a,b,c,d \in \ZZ\,$.  On $KT^4$, global one-forms can be written as 
\begin{align*}
\ea = d\xa\, , \quad \eb = d\xb\, , \quad \ec = d\xc\, , \quad \ed= d\xd + \xb d\xc\,.  
\end{align*}
We will take the symplectic structure on $KT^4$ to be
\begin{align}\label{ktsymp}
\om = \ea \w \eb + \ec \w \ed.
\end{align}
It is worthwhile to point out that $KT^4$ can be interpreted as a torus bundle over a torus in two ways:
\begin{align*}
\xymatrix{
T^2_{\{\xc, \xd\}} \ar[r] & KT^4 \ar[d] \\
 & T^2_{\{\xa, \xb\}} 
}\qquad\qquad\qquad
\xymatrix{
T^2_{\{\xa, \xd\}} \ar[r] & KT^4 \ar[d] \\
 & T^2_{\{\xb, \xc\}} 
}
\end{align*}
On the left, the fiber torus over a point on the base is a symplectic submanifold with respect to the $\om$ of \eqref{ktsymp}.  In contrast, the fiber torus on the right is a Lagrangian.  In fact, it is actually a special Lagrangian with respect to the global (2,0)-form $\Omega^{2,0} = (\ea + i\, \eb) (\ec + i\,\ed)$.

\begin{figure}[t]
\begin{tabular}{c|c|c|c|c|c|c}
$j$& $0$ & $1$ & $2$ & $3$ & $4$ & $5$ \\
\hline
& & $e_1$,   & $\om$ & $\om \, e_1,$ &  & \\
$H^j(KT^4)$ & $1$ & $e_2$, &$e_{12}-e_{34}$,  &$\om\, e_2,$  & $\om^2$&\\
& & $e_3$ & $e_{13}, e_{24}$ & $\om\, e_4 $ & & \\
\hline
&  & $e_1$,   & $e_{12}-e_{34},$ &$e_{12}-e_{34},$&$e_1$,  & \\
$F^0\!H^j(KT^4)$&$1$&$e_2$,&$e_{13}, e_{24}$, &$e_{13}, e_{24}$, &$e_2$,&$1$\\
&&$e_3$&$e_{14}$&$e_{23}$&$e_4$&\\
\hline
 &  &  $e_1$,   & $e_{12}-e_{34},$&$\theta(e_{12}-e_{34}),$&$\theta(\om\, e_{1})$,  & \\
$H^j(E)$&$1$&$e_2$,&$ e_{13}, e_{24},$&$\theta e_{13}, \theta e_{24}$&$\theta( \om\, e_{2})$&$\theta \om^2$\\
&&$e_3$&$e_{14}+\theta e_3$&$\theta e_{23}$&$\theta(\om\, e_{4})$&
\end{tabular}
\caption{Basis of generators for the cohomologies of the Kodaira-Thurston four-manifold $KT^4$ and its circle bundle $E$ (where $d\theta = \om$).  Here we use the notation $e_{ij} = e_ie_j$ and $e_{ijk}=e_i e_j e_k\,$.}
\label{Table.KT}
\end{figure}

To define $E$, the circle bundle over $KT^4$, we define the global angular one-form
\begin{align}\label{kttheta}
\theta = dy + \xa d\xb - \xd d\xc \,.
\end{align}
where $y$ is the coordinate on $S^1 = \RR/\ZZ$ with $y \sim y+1$.  Clearly, $d\theta = \omega$ on $E$.  Moreover, that $\theta$ is globally-defined implies that the circle fiber is twisted over $KT^4$ with the identification
\begin{align}\label{qeid} 
(\xa, \xd, y) \sim (\xa + a, \xd + d, y - a \,  \xb + d \, \xc ) 
\end{align}
where $a, d \in \ZZ$. 

In Table \ref{Table.KT}, we give a basis of generators for the de Rham cohomology and the $p=0$ filtered cohomology of $KT^4$, and also the de Rham cohomology of its circle bundle.  They have a heuristic dual chain description when promoted to currents.  The two special generators in $F^0H^*(M)$ that are distinct from those in $H^*(M)$ are $e_{14} \in F^0H^2(M)$ and $e_{23} \in F^0H^3(M)$.   They are both dual to Lagrangians.
\enum
\item $e_{14}\in F^0H^2(KT^4)$ corresponds to the dual current of a Lagrangian $L$ spanning the torus base in the $\{x_2, x_3\}$ directions.  Note that $d\, e_{14} = - e_{123}$, hence, the Lagrangian $L$ has a non-trivial boundary, i.e. $\partial L \neq 0\,$.  However, under the lift to $e_{14} + \theta e_3\in H^2(E)$, the corresponding three-dimensional cycle $\tilde L \subset E$ has no boundary.  
\item $e_{23}\in F^0H^3(KT^4)$ is the dual current of a Lagrangian $L'$ wrapping the torus fiber spanning $\{x_1, x_4\}$ over a point on the base.  Since $e_{23} = de_4$, $L'$ is rationally the boundary of a three-dimensional subspace. 
\item Note that the intersection of $L$ and $L'$ can be defined on $M$. The value corresponds to the standard intersection (even though $L$ is not a cycle).        
\enumd
\end{subsection}

\section{Functoriality of \texorpdfstring{$\cF$}{F}}
\label{section.functoriality}
As mentioned in the introduction, $\cF$ (or, equivalently, $\cC$) as an $A_\infty$-algebra depends only on the cohomology class $[\omega] \in H^2$, as opposed to $\omega$ itself.\footnote{See Corollary~\ref{cor.cohomology-class}, and also Remark~\ref{remark.C-infty-linear}. Of course, much math and physics is still done over the moduli of {\em cohomology classes} of symplectic or K\"ahler structures, so there is still interesting geometry to be explored.} However, invariants of manifolds are not useful merely for what they assign to manifolds, but for how they behave under morphisms (and families of morphisms). In this section we study the functoriality of the assignment $M \mapsto \cF(M)$, culminating in Proposition~\ref{prop.functor}. Note, however, that we fall short of obtaining how $\cF$ or $\cC$ behave in {\em families} of morphisms---this is because it seems that one needs a more sensible notion of homotopies between algebra maps in the $C^\infty$ category.

\begin{remark}\label{remark.C-infty-linear}
In this section, we treat $\cF(M)$ as an $A_\infty$-algebra over $\RR$, not as a smooth object living over $M$. The reader may observe that this is a loss of information---after all, both $\cC$ and $\cF$ are clearly local objects over $M$. There are at least three languages to convey what we mean: One could think of the equivalence $\cC \simeq \cF$ as an equivalence of (i) {\em algebroids} over $M$, rather than just algebras over $\RR$, or of (ii) dg bundles over $M$ with $A_\infty$ structures on their sections, or (iii) a sheaf of $C^\infty$-algebras over $M$ in the sense, for example, of Spivak~\cite{spivak}. We agree with the reader. However, at present, we have not yet worked out the full functoriality of these objects as $C^\infty$ objects living over $M$---as a simple example, when we explore how $\cF$ interacts with Lagrangian submanifolds, we will see that for non-transverse Lagrangian intersections, one naturally desires a derived intersection formula for algebroids. This technology has not yet been fully developed as far as we know, so we do not delve into it in this paper.
\end{remark}

\subsection{On objects}

Here is a corollary of Theorem~\ref{thm.cone}, which further illustrates the topological invariance of $\cF$ and $\cC$:

\begin{corollary}\label{cor.cohomology-class}
Let $\omega$ and $\omega'$ be two symplectic forms on $M$, and let $\cF_p(\omega)$ and $\cF_p(\omega')$ denote the $A_\infty$-algebras associated to each. If $[\omega] = [\omega'] \in H^2(M;\RR)$, then there is an equivalence of $A_\infty$-algebras $\cF_p(\omega) \simeq \cF_p(\omega')$ for every $p$.
\end{corollary}

\begin{proof}
It suffices to show an equivalence $\cone(\omega^{p+1}) \simeq \cone(\omega'^{p+1})$. More generally, let $\zeta$ and $\zeta'$ be two even-degree elements of a cdga that define the same cohomology class. We write an element of $\cone(\zeta)$, and its differential, as
	\eqnn
		\alpha + \theta \beta, \qquad d(\alpha + \theta \beta) = d\alpha + \zeta \beta \oplus -\theta \beta
	\eqnnd
and likewise for $\cone(\zeta')$:
	\eqnn
		\alpha + \theta' \beta, \qquad d(\alpha + \theta' \beta) = d\alpha + \zeta'\beta \oplus -\theta' \beta.
	\eqnnd
Let $\eta$ be an element of the cdga such that $d\eta = \zeta' - \zeta$. Then consider the map
	\eqnn
		\phi: \alpha \oplus \theta \beta \mapsto (\alpha -\eta \beta) \oplus \theta' \beta.
	\eqnnd
This is an equivalence of cdgas---one can easily check it is both a chain map and an algebra map, and an inverse chain map is given by
	\eqnn
		\alpha \oplus \theta' \beta \mapsto (\alpha + \eta \beta) \oplus \theta \beta.
	\eqnnd
\end{proof}

\begin{remark}
In particular, this shows that those invariants from \cite{tseng-yau}, \cite{tseng-yau-2}, \cite{tsai-tseng-yau} expressed entirely in terms of the $A_\infty$-equivalence class of $\cF_p$ (for example, the filtered cohomologies) are invariants only of the cohomology class of $\omega$. We offer one philosophical consistency check: Many invariants of symplectic geometry rely on a {\em positivity} condition arising from how $\omega$ evaluates on surfaces mapping to $M$. Since $\cF_0$ by definition only sees those forms which have no $\omega$ factor, these complexes are blind to measures of positivity.
\end{remark}

The proof of the following is straightforward:

\begin{prop}\label{prop.natural}
For all $p \geq 0$, the natural maps
	\eqnn
		\cone(\omega^{p+1}) \xra{q} \cone(\omega^p),
		\qquad
		q(\alpha \oplus \theta \beta) = \alpha \oplus \theta \beta \omega.
	\eqnnd
are maps of cdgas.
\end{prop}

As a result, we can fill in the dotted arrow below by composing the solid arrows, which are all maps of $A_\infty$-algebras:
	\eqnn
		\xymatrix{
			\cone(\omega^{p+1}) \ar[r]^q & \cone(\omega^p) \ar[d]^f \\
			\cF_p \ar[u]^g \ar@{-->}[r]^q & \cF_{p-1}.
		}
	\eqnnd

\begin{corollary}
For all $p \geq 0$, there are natural maps of $A_\infty$-algebras
	\eqnn
		\cF_p \to \cF_{p-1}.
	\eqnnd
\end{corollary}

Thus to every manifold $M$ we can associate a homotopy-commutative diagram of $A_\infty$-algebras as follows:
	\eqnn
		\xymatrix{
		\ldots \ar[r] \ar[d]_f^{\rotatebox{90}{$\sim$}}
			&\cone(\omega^{p+1}) \ar[r]^q \ar[d]_f^{\rotatebox{90}{$\sim$}}
			& \cone(\omega^p) \ar[d]_f^{\rotatebox{90}{$\sim$}}  \ar[r]
			&\ldots \ar[r] \ar[d]_f^{\rotatebox{90}{$\sim$}}
			& \cone(\omega) \ar[d]_f^{\rotatebox{90}{$\sim$}}\\
		\ldots \ar[r] 
			& \cF_p  \ar[r]^q 
			& \cF_{p-1} \ar[r]
			&\ldots \ar[r]
			& \cF_0
		}
	\eqnnd
One likewise has a homotopy commutative diagram where each downward $f$ is replaced by an upward-pointing $g$.

\begin{remark}
Recall from commutative algebra that for any primitive ideal $(f) \subset R$, one always has a sequence of commutative ring maps
	\eqnn
		R/ (f^{p+1}) \to R/(f^{p})
	\eqnnd
which cut out higher-degree infinitesimal neighborhoods around the locus $f = 0$. The above results, particularly Proposition~\ref{prop.natural}, are a cdga analogue. So the sequence of algebra maps (\ref{eqn.algebra-maps}) may be thought of as a sequence of neighborhoods around the locus where $\omega \in \Omega^2(M)$ vanishes.\footnote{Not as a locus on $M$, but as a locus on {\em stacky points} on the stack represented by $\cone(\omega)$.}
\end{remark}

\subsection{On morphisms} 
Note that if $f: (M,\omega) \to (M',\omega')$ is any smooth map such that $f^*\omega' = \omega$, one has an induced map of cdgas
	\eqn\label{eqn.cone-pullback}
		f^*_{\cone}: \cone(\omega'^{p+1}) \to \cone(\omega^{p+1}),
		\qquad
		\alpha \oplus \theta' \beta \mapsto f^*\alpha \oplus \theta f^*\beta.
	\eqnd
Here, $f^*\alpha$ and $f^*\beta$ are the usual pullbacks of differential forms. 

Now let us elaborate on the naturality asserted in Proposition~\ref{prop.natural}. By naturality, we mean that given any smooth map $f: M \to M'$ such that $f^*\omega' = \omega$, the diagram
	\eqnn
		\xymatrix{
			\cone(\omega'^{p+1})\ar[d]^{ f^*_{\cone}}	\ar[r]^q 
				&\cone(\omega'^p) \ar[d]^{ f^*_{\cone}} \\
			\cone(\omega^{p+1}) \ar[r]^q 
				&\cone(\omega^p) 
		}
	\eqnnd
commutes on the nose. What this shows is that maps between symplectic manifolds induce maps between the sequences of algebras $(\cone(\omega^{p+1}))_{p \geq 0}$. So our work so far has lead to the following functorial description:

Let $\cdgadiscrete$ be the category whose objects are cdgas, and whose morphisms are maps of cdgas. We let $\ZZ_{\geq 0}$ denote the usual poset, considered as a category. We let $\fun(\ZZ_{\geq0}^{\op}, \cdgadiscrete)$ denote the functor category, whose morphisms are natural transformations. Finally, let $\sympdiscrete$ denote the category whose objects are symplectic manifolds, and whose morphisms are smooth maps respecting the symplectic forms.

\begin{prop}\label{prop.functor}
The assignments 
	\begin{itemize}
		\item
			$M \mapsto \left(\ldots \to \cone(\omega^{p+1}) \to \cone(\omega^p) \to \ldots \to \cone(\omega)\right)$ and
		\item
			$f \mapsto f^*_{\cone}$ from (\ref{eqn.cone-pullback})
	\end{itemize}
define a functor
	\eqnn
		\sympdiscrete \to \fun(\ZZ_{\geq 0}^{\op}, \cdgadiscrete).
	\eqnnd
\end{prop}

\begin{proof}
Identities are obviously sent to identities, since $\id^*\alpha = \alpha$ for differential forms. Composition is respected by the definition of $f^*_{\cone}$ in (\ref{eqn.cone-pullback}). 
\end{proof}

\subsection{On orthogonal homotopies}
Just as de Rham forms respect smooth homotopies, it is natural to ask whether $\cone(\omega^{\bullet+1})$ respects certain kinds of homotopies between symplectic maps. While the assignment respects other kinds of smooth homotopies if one just remembers $\cone(\omega^{\bullet +1})$ as a cdga, we have purposefully restricted ourselves to a class of smooth homotopies for which the homotopy formulas can be expressed in terms of local operators---i.e., for which we can stay in the $C^\infty$ world.

\begin{defn}\label{defn.orthogonal}
Let $F: M \times \Delta^1 \to M'$ be a smooth map, and let $\pi: M \times \Delta^1 \to \Delta^1$ be the projection. Endow both $M$ and $M'$ with symplectic forms. We will say that $F$ is an {\em orthogonal homotopy of symplectic maps}, or orthogonal homotopy for short, if
	\eqn\label{eqn.orthogonal}
		F^*\omega' = \pi^*\omega .
	\eqnd
\end{defn}

\begin{remark}
The condition of being an orthogonal homotopy is a rigid one. The condition implies that
	\eqnn
		F^*\omega'(\del_t, v) = \omega'(DF(\del_t),DF(v))
	\eqnnd
for any $v \in TM$. Hence for any time $t$, the vector $DF(\del_t)$ must always be in the symplectic orthogonal to the image of $F_t(M)$. In particular, a homotopy obtained by pre-composing a symplectic immersion $j: M \to M'$ by a symplectic isotopy of $M$ is not an example of an orthogonal homotopy unless the isotopy is trivial (i.e., constant).
\end{remark}

\begin{example}
If one is given a fibration $E \to B$ where $E$ is symplectic and the symplectic form renders each fiber symplectic, then the symplectic orthogonals to each fiber determine a horizontal distribution. Parallel transport defines an orthogonal homotopy.
\end{example}

\begin{prop}\label{prop.homotopy}
If $F: M \times \Delta^1 \to M'$ is an orthogonal homotopy, then $F$ induces a homotopy of chain maps between $(F_0)^*_{\cone}$ and $(F_1)^*_{\cone}$.
\end{prop}

Recall how one normally proves that a homotopy $F: M \times \Delta^1 \to M'$ induces a homotopy of chain maps between $F_0^*$ and $F_1^*$ on differential forms: One defines a degree -1 map
	\eqn\label{eqn.H}
		H: \Omega^\bullet(M') \to \Omega^{\bullet - 1}(M),
		\qquad
		H(\alpha') := \int_0^1 (\iota_{\del_t} F^*\alpha') dt
	\eqnd
where $t$ is a coordinate for $\Delta^1$. Then $H$ satisfies 
	\eqnn
		H d + d H = F_1^* - F_0^*
	\eqnnd
and hence exhibits a chain homotopy from $F_1^*$ to $F_0^*$. 

So, given a smooth homotopy $F: M \times \Delta^1 \to M'$, define a degree -1 map as follows:
	\eqn\label{eqn.H-tilde}
		\tilde H: \cone(\omega')^\bullet \to \cone(\omega)^{\bullet-1},
		\qquad
		\alpha \oplus \theta'\beta
		\mapsto
		H(\alpha) \oplus -\theta H(\beta).
	\eqnd

\begin{lemma}\label{lemma.H-tilde}
Let $F: M \times \Delta^1 \to M'$ be a smooth homotopy, and let $\tilde H$ be the operator defined in (\ref{eqn.H-tilde}).
$\tilde H$ exhibits a homotopy between the pullbacks $(F_0)^*_{\cone}$ and $(F_1)^*_{\cone}$ as defined in (\ref{eqn.cone-pullback}) if and only if one has
	\eqn\label{eqn.orthogonal-1}
		H (\omega'\beta) = \omega H(\beta).
	\eqnd
\end{lemma}

\begin{proof}[Proof of Lemma~\ref{lemma.H-tilde}.]
We see that
	\begin{align}
		\tilde H d +d\tilde H( \alpha \oplus \theta'\beta)
		&= \tilde H (d\alpha + \omega'\beta \oplus -\theta'd\beta) + d(H\alpha \oplus -\theta H(\beta)) \nonumber \\
		&= (Hd\alpha \oplus \theta Hd\beta) + H(\omega'\beta) \oplus 0
			+ dH\alpha \oplus \theta dH(\beta) -\omega H(\beta) \oplus 0 \nonumber\\
		&= (Hd + dH)\alpha \oplus \theta (Hd + dH)\beta + \left( H(\omega'\beta) - \omega H( \beta)\right) \oplus 0 \nonumber \\
		&= (F^*_1 - F^*_0) \alpha \oplus \theta (F_1^* - F_0^*)\beta + \left( H(\omega'\beta) - \omega H( \beta)\right) \oplus 0  \nonumber \\
		&= \left((F_1)^*_{\cone} - (F_0)^*_{\cone}\right)( \alpha \oplus \theta \beta) + \left( H(\omega'\beta) - \omega H( \beta)\right) \oplus 0   .\nonumber
	\end{align} 
\end{proof}

\begin{proof}[Proof of Proposition~\ref{prop.homotopy}.]
Writing out (\ref{eqn.H}), the lefthand side of (\ref{eqn.orthogonal-1}) becomes
	\begin{align}
		H(\omega'\beta)
		&= \int_0^1 \iota_{\del_t} F^*(\omega'\beta) dt \nonumber \\
		&= \int_0^1 (\iota_{\del_t}F^*\omega') \wedge F^*\beta + F^*\omega' \wedge \iota_{\del_t}F^*\beta dt \nonumber\\
		&= \int_0^1 (\iota_{\del_t}F^*\omega') \wedge F^*\beta + F^*\omega' \wedge \iota_{\del_t}F^*\beta dt \nonumber
	\end{align}
while the righthand side  of (\ref{eqn.orthogonal-1}) is given by
	\begin{align}
		\omega H(\beta)
		&= \omega \int_0^1 \iota_{\del_t}(F^*\beta) dt. \nonumber
	\end{align}
Now we show that a sufficient condition for (\ref{eqn.orthogonal-1}) to hold is the given by (\ref{eqn.orthogonal}). 

The condition (\ref{eqn.orthogonal}) implies that for every time $t$, $F_t^* \omega' = \omega$. This is because, letting $i_t: M \into M \times \Delta^1$ be the inclusion at time $t$, we have that
	\eqnn
		(F \circ i_t)^* \omega' 
		= i_t^* F^* \omega'
		= i_t^* \pi^*\omega
		= (\pi \circ i_t)^*\omega
		= \id_M^*\omega
		= \omega.
	\eqnnd
In other words, $F$ is a homotopy through maps that respect the symplectic form. Even more strongly, the definition implies that for all $v \in TM$, we have
	\eqnn
		(\iota_{\del_t}F^*\omega')v
		= F^*\omega'(\del_t, v)
		= \pi^*\omega (\del_t,v)
		= \omega ( D\pi(\del_t), D\pi(v))
		= \omega(0,D\pi(v))
		= 0.
	\eqnnd
Hence the definition implies that
	\eqnn
		\iota_{\del_t} F^*\omega' = 0.
	\eqnnd
Because $F_t^*\omega' = \omega$ for all $t$, this is enough for (\ref{eqn.orthogonal-1}) to hold.
\end{proof}

\begin{remark}
One would like to further say that orthogonal homotopies determine homotopies between the {\em algebra} maps, rather than just the chain maps. However, we have not delved into the algebraic theory for homotopies in the smooth category for the reasons mentioned in Remark~\ref{remark.C-infty-linear}; for instance, the usual definition for a homotopy between cdga maps uses polynomial forms on the 1-simplex, which are only suitable when one deals with smooth homotopies whose dependence on $t$ happens to be polynomial. (I.e., not very often.) 
\end{remark}

\subsection{Sheaf property}
Note that being primitive is a local property. Moreover, $\del_-$ and $\del_+$ are local operators, so the differentials of $\cF$ are compatible with the restriction of differential forms from one open subset to another. All this data forms a sheaf:

\begin{theorem}\label{thm.sheaf}
The assignment $U \mapsto \cF(U)$ is a homotopy sheaf of $A_\infty$-algebras on $M$. 
\end{theorem}

We recall that if $\cC$ is an $\infty$-category, a presheaf on $M$ with values in $\cC$ is a functor
	\eqnn
		\cF: N(\open(M))^{\op} \to \cC.
	\eqnnd
Here, $\open(M)$ is the category of open subsets of $M$, and $N$ is its nerve (i.e., the associated $\infty$-category). We call $\cF$ a sheaf is the following holds: For any open cover $\{U_\alpha \}$ of any open set $U$, the associated augmented cosimplicial diagram
	\eqnn
		\xymatrix{
		\cF(U) \ar[r] & \prod_\alpha \cF(U_\alpha) \ar@2[r] & \prod_{\alpha,\beta} \cF(U_\alpha \cap U_\beta) \ar@3[r] & \ldots
		}
	\eqnnd
is a limit diagram. 

\begin{proof}
Let $\sA_\infty\Alg$ be the $\infty$-category of $A_\infty$-algebras over $\RR$. Then the assignment $U \mapsto \cF(U)$ defines a functor
	\eqnn
		\cF: N(\open(M))^{\op} \to \sA_\infty\Alg.
	\eqnnd
Moreover, since the equivalences $\cF \to \cC$ are made of local operators (hence compatible with restriction maps), it suffices to show that $U \mapsto \cC(U)$ is a homotopy sheaf. We do this now.

Since the forgetful functor $\sA_\infty\Alg \to \chain$ creates limits, one need only prove that this augmented diagram is a limit diagram in the $\infty$-category of cochain complexes. But in general, if $\cF$ is a sheaf (in the classical sense) of cochain complexes in which the degree $n$ presheaf $\cF^n$ is a soft sheaf for every $n$, then $\cF$ is itself a homotopy sheaf. That is, the augmented diagram in cochain complexes is a limit diagram.

So one simply needs to check that $\cC$ is a complex of soft sheaves. This is elementary---each $\cC^k$ is soft because differential forms on a closed subset can extend to a differential form globally. $\cC$ is a sheaf for the same reasons that differential forms form a sheaf.
\end{proof}

\begin{remark}
In fact, if one considers the $\infty$-category
	\eqnn
		\fun(\ZZ_{\geq 0}^{\op},\sA_\infty\Alg)
	\eqnnd
then we see that $\cF$ defines a sheaf on $M$ with this target $\infty$-category---this is because limits in a diagram category are computed pointwise (i.e., one need only check that for every $n \in \ZZ_{\geq 0}$, the resulting cosimplicial diagram is a limit diagram). Note that $\sA_\infty\Alg$ contrasts with $\underline{\cdga}$, which did not incorporate higher homotopies. And if one uses the model of filtered forms using the cone construction---utilizing the equivalence of Theorem~\ref{thm.cone}---then one can make a stronger statement: The assignment 
	\eqnn
		U \mapsto \left(\ldots \to \cone(\omega^{p+1})(U) \to \cone(\omega^p)(U) \to \ldots \to \cone(\omega)(U)\right)
	\eqnnd
is a homotopy sheaf in the $\infty$-category $\fun(\ZZ_{\geq 0}^{\op}, \cdga)$ where we remove the underline.
\end{remark}

\subsection{As sheaf cohomology}
Recall Leray's theory of sheaf cohomology, which in \hyphenation{Gro-then-dieck}Grothendieck's language is obtained as a right derived functor of the global sections functor $F \mapsto F(M)$. 
It is a consequence of the de Rham-Weil theorem that these cohomology groups can be computed by taking an acyclic (rather than injective) resolution.

One sees immediately that the de Rham forms $\Omega^k(M)$ are soft, as we mentioned in the proof of Theorem~\ref{thm.sheaf}. Likewise, it is clear that each sheaf $\cF_p^k$ or $\cone(\omega^{p+1})^k$ is a soft sheaf on $M$. So we must ask---is there a natural sheaf for which these complexes are a resolution?

Though the following discussion is valid using the $A_\infty$-algebra $\cF$ as well, we will focus on the sheaf $\cone(\omega^{p+1})$. The following should be compared to Proposition~3.3 of~\cite{tsai-tseng-yau}:

\begin{prop}\label{prop.poincare}
Locally on $M$, the sequence
	\eqnn
		\Omega^{k-1} \oplus \theta \Omega^{k-2p-2}
		\to
		\Omega^k \oplus \theta \Omega^{k-2p-1}
		\to
		\Omega^{k+1} \oplus \theta \Omega^{k-2p}
	\eqnnd
is exact for $k \geq 2p+2$.
\end{prop}

\begin{proof}
Evaluate the sequence on some open set $U$ which is diffeomorphic to $\RR^{2n}$. If $\alpha\oplus \theta \beta$ is in the kernel of $d|_{U}$, then we know $d\alpha = -\omega^{p+1} \beta$ and $d\beta = 0$. The Poincar\'e Lemma says we can find $\beta'$ such that $-d\beta' = \beta$. Moreover, $\alpha - \omega^{p+1}\beta'$ is closed, so the Poincar\'e Lemma also allows us to find $\alpha'$ so that $d\alpha' = \alpha - \omega^{p+1} \beta'$. Thus $\alpha\oplus\theta\beta = d(\alpha'\oplus\theta\beta')$. 
\end{proof}

The exactness fails at the $2p+1$ group of $\cone(\omega^{p+1})$. Of course, the complex is exact at degrees $1 \leq k \leq 2p$ by the usual Poincar\'e Lemma. For this reason, the sheaf cohomology description is cleanest when $p=0$, though one can easily formulate the analogue for higher $p$.

\begin{defn}
For any open set $U \subset M$, let $\prim'(U)$ be the vector space whose elements are pairs $(\alpha, b)$ where $b$ is a locally constant function on $U$, and $\alpha$ is a 1-form such that $d\alpha = -b\omega|_U$. We let $\prim$ denote the sheafification of $\prim'$.
\end{defn}
 
\begin{prop}
$\cone(\omega^p)$ is a soft resolution of the two-term complex of sheaves
	\eqnn
		\Omega^0 \to \prim,
		\qquad
		g \mapsto (dg, 0).
	\eqnnd
\end{prop}

\begin{proof}
We simply need to show that the inclusion
	\eqnn
		\xymatrix{
			\Omega^0 \ar[r]^d \ar[d] & \prim \ar[r]\ar[d]  & 0 \ar[d] \ar[r] & \ldots\ar[d]  \\
			\Omega^0 \ar[r]^-d & \Omega^1 \oplus \theta \Omega^0 \ar[r] & \Omega^2 \oplus \theta \Omega^1 \ar[r] & \ldots
		}
	\eqnnd
is an equivalence of complexes of sheaves---that is, that the map induces an isomorphism of cohomology sheaves $\cH^k$. We already saw in Proposition~\ref{prop.poincare} that the $\cone(\omega)$ complex is indeed locally acyclic at $k \geq 2$. So one need only show an isomorphism at $\cH^0$ and $\cH^1$. The isomorphism at $\cH^0$ is obvious, as the inclusion naturally identifies the kernels of $d$---the locally constant functions. The isomorphism at $\cH^1$ is also obvious, as $\prim'$ is defined to locally be the kernel of the differential $\Omega^1 \oplus  \theta \Omega^0 \to \Omega^2 \oplus \theta \Omega^1$. 
\end{proof}

\begin{remark}
A similar proof shows the trivial result that $\cone(\omega^{p+1})$ is a soft replacement for the complex of sheaves
	\eqnn
		\Omega^0 \to \ldots \to \Omega^{2p} \to \ker(d|_{\Omega^{2p+1} \oplus \theta \Omega^0}) \to 0
	\eqnnd
which, of course, is soft away from the last non-zero entry.
\end{remark}

\subsection{Looking forward: Weinstein functoriality}
Each sequence of cdgas $\cone(\omega^{\bullet+1})$ depends only on the cohomology class of $\omega$. However, the equivalence $\cone(\omega) \simeq \cone(\omega')$ from Corollary~\ref{cor.cohomology-class} relied on a choice of $\eta$ such that $d\eta = \omega' - \omega$. At this point, we see two doorways into further structure:
\enum
	\item
		The first appears when we recall that maps which respect $\omega$ are only one kind of morphism between symplectic manifolds. As advocated by Weinstein, a more natural notion of morphism is given by submanifolds of $M \times M'$. In general, graphs of symplectomorphisms give rise to Lagrangian submanifolds of $M \times M'$, but the fact that $\cone(\omega^{\bullet+1})$ is functorial with respect to all maps respecting $\omega$ suggests the following: {\em isotropic} submanifolds of $M \times M'$ should also be included in the class of morphisms we consider. (Graphs of smooth maps respecting $\omega$ are, in general, isotropic.) In other words, we should look for a Weinstein functoriality with respect to isotropic correspondences between symplectic manifolds---not just smooth maps respecting $\omega$.
	\item
		The second is that {\em specifying} forms $\eta$ that realize the equality $[\omega] = [\omega']$ is a natural piece of data to include in whatever kind of functoriality we consider. This is a picture one has seen before: For instance, in the framework of~\cite{ptvv}, derived symplectic manifolds and their Lagrangians all carry with them {\em additional data} showing how symplectic forms are closed, and how their restrictions to Lagrangians are null. 
\enumd

Following through on the second doorway goes beyond the scope of our current paper, as we would need to seriously develop the machinery for derived smooth manifolds, for their cotangent complexes, and for writing down a stack classifying closed 2-forms. (Roughly, we would need a $C^\infty$ analogue of all the basic ingredients set up in~\cite{ptvv} in the context of derived {\em algebraic} geometry.) So here we merely open the first door.

Let $j: L \to M_1 \times M_2$ be an immersion. Also assume that $j^*\omega_1 = j^*\omega_2$, so that $L$ is an immersed isotropic submanifold of $(M_1 \times M_2, -\omega_1 \oplus \omega_2)$. 

\begin{defn}
We let $\cone_L(\omega^p)$ denote the mapping cone of the map
	\eqnn
		\Omega(L)[-2p] \xra{j^*\omega_2} \Omega(L).
	\eqnnd
\end{defn}

\begin{remark}
One should expect a construction of $\cone_L$ having the same flavor as the construction of the $\cF_p$. We do not know of any such construction at present.
\end{remark}

Because $j$ is an isotropic immersion, we see immediately:

\begin{prop}
Let $\cone_{M_i}(\omega_i^p)$ denote the cone algebra for $M_i$. The composite smooth maps $L \to M_1 \times M_2 \to M_i$ induce maps $\Omega(M_i) \to \Omega(L)$, and these in turn induce maps of cdgas
	\eqnn
		\cone_{M_1}(\omega_1^p) \to \cone_L(\omega^p) \leftarrow \cone_{M_2}(\omega_2^p).
	\eqnnd
In other words, $\cone_L(\omega^p)$ is a bimodule for each cone algebra.
\end{prop}

In particular, any Lagrangian immersion defines a bimodule between the algebras associated to each $M_i$. Note also that if $j:L \to M_1 \times M_2$ is not strictly an isotropic map, but if one can specify a form $\eta_L$ such that $d\eta_L = j^*\omega_2 - j^*\omega_1$, then this specifies an equivalence
	\eqnn
		\cone_L(\omega_1) \simeq \cone_L(\omega_2)
	\eqnnd
hence one still has a bimodule over the two cone algebras given by the $M_i$.
This is exactly the kind of data that we should witness when following through with door (2), which we leave for later work.

Finally, we also leave for later work the proof of Weinstein functoriality in (1), which likewise requires a development of some derived smooth geometry to incorporate non-transverse intersections of Lagrangian and isotropic correspondences. The expectation, of course, is that compositions of Lagrangian correspondences are taken to tensor products of bimodules under the functor $\cF_\bullet$ (or equivalently, the functor $\cone(\omega^{\bullet+1})$).

\section{Appendix: Proof that \texorpdfstring{$\cone(\omega)$}{Cone(omega)} is equivalent to differential forms on a sphere bundle}
Assume $[\omega]$ is an integral cohomology class on $M$. Then $[\omega]$ classifies a complex line bundle $L$ on $M$, and hence a circle bundle. More generally, the higher powers $[\omega^{p+1}]$ are the top Chern classes of the vector bundles $L^{\oplus p+1}$. There is a natural smooth sphere bundle $E_p$ associated to this vector bundle. 

In this appendix, we give a proof of the following: 

\begin{theorem}\label{theorem.appendix}
Let $M$ be an arbitrary symplectic manifold, and $\omega$ a symplectic form defining an integral cohomology class. Then there is an equivalence of cdgas
	\eqnn
		\Omega(E_p) \simeq \cone(\omega^{p+1}).
	\eqnnd
\end{theorem}

\begin{remark}
This result is obvious if $M$ is simply-connected---by rational homotopy theory, the cone is a standard cdga model for forms on the total space of an odd-dimensional sphere bundle. In fact, so long as $M$ is simply-connected, the result also holds for $[\omega]$ an arbitrary {\em real} cohomology class, so long as $E_p$ is replaced by the fibration classified by the real cohomology class $[\omega^{p+1}]$. (Such an $E_p$ has fibers equivalent to the Eilenberg-MacLane spaces $K(2p+1;\RR)$.)
\end{remark}

\begin{remark}
Fix a real number $k\neq 0$. Note that if one scales $\omega$ to $k\omega$, one has an isomorphism of cdgas as follows:
	\eqnn
		K: \cone(\omega^{p+1}) \simeq \cone(k^{p+1}\omega^{p+1}),
		\qquad
		\eta \oplus \theta \zeta
		\mapsto
		\eta \oplus {\frac 1 {k^{p+1}}} \theta \zeta.
	\eqnnd
\end{remark}

The proof of the theorem uses the following easy lemma:

\begin{lemma}
Theorem~\ref{theorem.appendix} holds when $M$ is contractible.
\end{lemma}

\begin{proof}[Proof of Lemma.]
When $M$ is contractible, $E_p$ is trivial, and homotopy equivalent to $S^{2p+1}$. Hence we can choose a global $2p+1$-form called $\underline{\theta}$ that realizes a generator for the cohomology of the fiber (and hence of $E_p$). Then we claim that the cdga map
	\eqn\label{eqn.cone-Omega-map}
		\cone(\omega^{p+1}) = \Omega^\bullet(M) \oplus \theta \Omega^{\bullet - 2p-1}(M)  \to \Omega(E_p),
		\qquad
		\eta \oplus \theta \zeta
		\mapsto
		\pi^*\eta + \underline{\theta}\wedge\pi^*\zeta
	\eqnd
exhibits the equivalence. Note it is obviously a surjection on cohomology, as the generators of $H^0(S^{2p+1})$ and $H^{2p+1}(S^{2p+1})$ are hit by
	\eqnn
	1 \oplus 0
	\qquad
	\text{and}
	\qquad
	0 \oplus \theta 1,
	\eqnnd
respectively. (Here, ``1'' is the constant 0-form with value 1.) It's further an injection on cohomology, because the long exact sequence in cohomology for a cone (along with the contractibility of $M$) shows that these are the only possible cohomology generators in $\cone(\omega)$. 
\end{proof}

\begin{proof}[Proof of Theorem~\ref{theorem.appendix}.]
$E_p$ is an oriented sphere bundle, being the sphere bundle associated to a complex vector bundle. We let $\underline{\theta} \in \Omega^{2p+1}(E_p)$ denote a global angular form on $E_p$. This $\theta$ can be chosen to satisfy two properties: (i) $\underline{\theta}$ restricts to a generator of $H^{2p+1}(S^{2p+1})$ on each fiber of $E_p$, and (ii) $d\underline{\theta} = \pi^*(\omega^{p+1})$. Note property (ii) follows because $\omega^{p+1}$ is a representative of the Euler class of $E_p$. These are classical facts one can find, for instance, in Bott and Tu~\cite{bott-tu}.

So the same map as in~(\ref{eqn.cone-Omega-map}) is a map of cdgas.

Let $\cV=\{V_\alpha\}$ be a good cover of $M$.
If $\pi: E_p \to M$ is the projection map of the fibration, the open sets $U_\alpha = \pi^{-1}(V_\alpha)$ form an open cover of $E_p$, where each non-empty intersection $U_{\alpha_0,\ldots,\alpha_r} := U_{\alpha_0} \cap \ldots \cap U_{\alpha_r}$ is homotopy equivalent to the fiber sphere.
Consider the \v{C}ech-de Rham double complex, which is the $E_0$ page of the spectral sequence given by the usual filtration on on $\Omega(E_p)$.
	\eqnn
	\xymatrix{
	\vdots
		& \vdots
		& \vdots
		&  \\
	\prod_\alpha \Omega^2(U_\alpha) \ar[u]^d \ar[r]^-\delta
		& \prod_{\alpha_0,\alpha_1} \Omega^2(U_{\alpha_0,\alpha_1}) \ar[u]^{-d} \ar[r]^-\delta
		& \prod_{\alpha_0,\alpha_1,\alpha_2} \ar[u]^d\Omega^2(U_{\alpha_0,\alpha_1,\alpha_2}) \ar[r]
		& \ldots \\
	\prod_\alpha \Omega^1(U_\alpha) \ar[u]^d \ar[r]^-\delta
		& \prod_{\alpha_0,\alpha_1} \Omega^1(U_{\alpha_0,\alpha_1}) \ar[u]^{-d} \ar[r]^-\delta
		& \prod_{\alpha_0,\alpha_1,\alpha_2} \ar[u]^d\Omega^1(U_{\alpha_0,\alpha_1,\alpha_2}) \ar[r]
		& \ldots \\
	\prod_\alpha \Omega^0(U_\alpha) \ar[u]^d \ar[r]^-\delta
		& \prod_{\alpha_0,\alpha_1} \Omega^0(U_{\alpha_0,\alpha_1}) \ar[u]^{-d} \ar[r]^-\delta
		& \prod_{\alpha_0,\alpha_1,\alpha_2} \Omega^0(U_{\alpha_0,\alpha_1,\alpha_2}) \ar[r]\ar[u]^d
		& \ldots
	}
	\eqnnd
The advantage of working in the smooth world (and hence using the de Rham model) is that this double complex has a total complex which is a cdga.\footnote{As opposed to having to keep track of $E_\infty$ structures on singular cochains.} Moreover, the usual augmentation from $\Omega^\bullet(E_p)$ gives a map of cdgas from $\Omega(E_p)$ to the total complex, and this is an equivalence of cdgas. (For instance, each row of this $E_0$ page is exact by the existence of partitions of unity. This is one classical way to see that the \v{C}ech-de Rham double complex computes the de Rham cohomology algebra of the total space.)

Now we note that the map of cdgas~(\ref{eqn.cone-Omega-map}) is local---it is compatible with restriction maps. Hence~(\ref{eqn.cone-Omega-map}) induces a map of double complexes. For the sake of having a name, we will call the domain double complex the ``$V$ double complex.'' At the $r$th column and $s$th row, the map from the $V$ double complex to the \v{C}ech-de Rham complex is given by
	\eqnn
		\Omega^s(V_{\alpha_0,\ldots,\alpha_r}) \oplus \theta \Omega^{s-2p-1}(V_{\alpha_0,\ldots,\alpha_r})
		\to \Omega^s(U_{\alpha_0,\ldots,\alpha_r}),
		\qquad
		\eta \oplus \theta \zeta \mapsto \pi^*\eta + \underline{\theta}\wedge\pi^*\zeta.
	\eqnnd
But this is the map used in the Lemma, hence an equivalence of cdgas. Thus the vertical differentials $d$ of the $V$ double complex compute the same $E_1$ page of the spectral sequence associated to the \v{C}ech-de Rham double complex. Importantly, the horizontal differentials of the $E_2$ page are also equal---this is where one uses the existence of the global angular form. Specifically, the horizontal \v{C}ech complexes have compatible differentials precisely because a single global form restricts to the generators along each open set.

The last (and only other) page at which one has a differential is the $E_{2p+2}$ page, where the differential is precisely given by the map called ``wedge with $\omega^{p+1}$.'' (This is by definition of the differential of $\cone(\omega^{p+1})$, which leads to this differential in the $V$ double complex, and by a well-known property of the Euler form for the \v{C}ech-de Rham double complex.)

This completes the proof.
\end{proof}

\printindex

\bibliographystyle{amsalpha}
\bibliography{bib}

\end{document}